\documentclass{amsart}
\usepackage[margin=1in]{geometry}
\usepackage{graphicx} 
\usepackage{enumerate,enumitem,verbatim}
\usepackage{xcolor,tikz-cd}
\usepackage{listings}
\usepackage[colorlinks=true,citecolor=black,linkcolor=black,urlcolor=blue]{hyperref}
\newcommand{\seqnum}[1]{\href{http://oeis.org/#1}{\underline{#1}}}
\usepackage{subcaption}
\usepackage{mathtools,cleveref}
\usepackage{amssymb,amsmath}
\RequirePackage{hyperref}
\usepackage{cleveref}
\usepackage[noend]{algpseudocode}
\usepackage[linesnumbered,ruled]{algorithm2e}

\theoremstyle{definition} 
\newtheorem{definition}{Definition}[section]
\newtheorem{theorem}{Theorem}[section]
\newtheorem{corollary}{Corollary}[theorem]

\newtheorem{lemma}{Lemma}[section]

\theoremstyle{remark}
\newtheorem{remark}[theorem]{Remark}

\newtheorem{question}{Question}

\newcommand{\Lean}{\mathrm{Lean}}
\newcommand{\Silean}{\mathrm{Silean}}
\newcommand{\Bilean}{\mathrm{Bilean}}
\newcommand{\cleave}{\mathrm{cleave}}

\usepackage{tikz}
\makeatletter
\newcommand{\vo}
{\vec{o}\@ifnextchar{^}{\,}{}}
\makeatother

\title{The Lean Number of a Hypergraph}

\author[Piehowski]{Harrison Piehowski}
\address[H.\ Piehowski]{Department of Mathematical Sciences, University of Wisconsin-Milwaukee, Milwaukee, WI 53211
}
\email{\textcolor{blue}{\href{mailto:piehows4@uwm.edu}{piehows4@uwm.edu}}}

\begin{document}

\begin{abstract}
Inspired by the notion of tricolorability of knots, we introduce the concept of \textit{lean coloring} for hypergraphs and the associated \textit{lean number} of a hypergraph. 
Lean coloring often involves very few colors, yet still requires the methods of usual graph coloring, forcing the overall complexity to be NP-Hard. We provide two alternative formulations of the lean coloring problem that involve a type of coloring on abstract simplicial complexes and a partial coloring on bipartite graphs. 
We then provide bounds for the lean numbers of hypergraphs that are $k$-uniform, $k$-partite, wide-path connected, or $r$-complete.
Python-like script is included to allow the implementation and study of a lean coloring algorithm. We conclude with some directions for future work and present the lean numbers of $130$ knots and links.
\end{abstract}

\maketitle

\section{Introduction}

We introduce lean coloring and the lean number of a hypergraph and provide results on this new graph coloring invariant. 
To begin, recall that an \textit{(undirected) hypergraph} $H=(V,E)$ is defined as a pair of sets such that $V \subset \mathbb{Z}^+$ (called the set of \textit{vertices}), $E \subseteq \mathcal{P}(V)$ (called the set of \textit{(hyper)edges}), where $\mathcal{P}(V)$ denotes the power set of $V$. 
Throughout, 
we assume that $2 \leq |V| < \infty$, $1 \leq |E| < \infty$, and that there are no empty edges, i.e., $\emptyset\notin E$.  
A \textit{lean coloring of $H$} is an assignment of \underline{at least two} colors to the vertices in $V$ so that within each edge, the vertices display either all distinct colors or exactly one color. Formally, a lean coloring $L$ is a function $L:V \to \mathbb{Z}^+$ so that $|L(V)| \geq 2$ and that for every $e \in E$: $|L(e)| \in\{1,|e|\}$. If $|L(e)| = 1$, we say that $e$ is \textit{monocolored} under $L$, and if $|L(e)| = |e|$, we say that $e$ is \textit{multicolored} under $L$. The \textit{lean number} of a hypergraph $H$ is the fewest colors necessary to lean color $H$, and we denote this number by $\Lean(H)$. \Cref{fig:leancolorings} illustrates three distinct lean colorings of a hypergraph.

\begin{figure}[h]
    \centering
    \begin{subfigure}[b]{0.25\textwidth}
        \includegraphics[width=\textwidth]{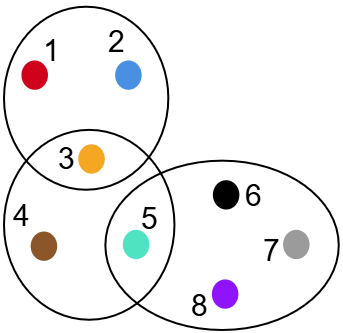}
        \caption{Coloring 1}
        \label{fig:color1}
    \end{subfigure}
    ~
    \begin{subfigure}[b]{0.25\textwidth}
        \includegraphics[width=\textwidth]{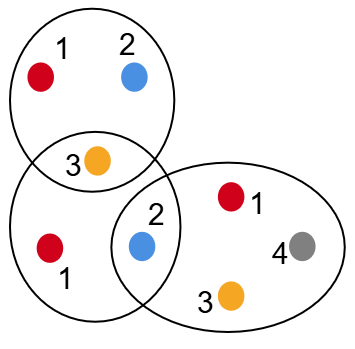}
        \caption{Coloring 2}
        \label{fig:color2}
    \end{subfigure}
    ~
    \begin{subfigure}[b]{0.25\textwidth}
        \includegraphics[width=\textwidth]{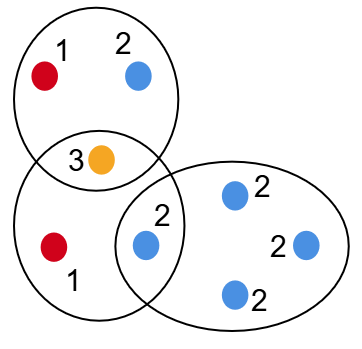}
        \caption{Coloring 3}
        \label{fig:color3}
    \end{subfigure}
    \caption{Three lean colorings of a hypergraph. 
    In addition to each vertex being colored, 
    the colors are associated with the numbers $1$ through $8$ to aid in distinguishing between the colors. 
    Coloring $1$ uses eight colors, 
    but this is not optimal, meaning that there is a lean coloring which uses fewer colors. 
    Coloring $2$ uses four colors, but is still not optimal. Under colorings $1$ and $2$, there are three multicolored edges. Coloring $3$ is optimal by inspection and demonstrates that the lean number of this hypergraph is $3$. Under coloring $3$, there are two multicolored edges and one monocolored edge. Optimal lean colorings of a hypergraph are far from unique.}
\label{fig:leancolorings}
\end{figure}

Note that coloring every vertex a distinct color results in a valid lean coloring; however, such a lean coloring is often not optimal. For example, coloring $1$ in \Cref{fig:leancolorings} uses eight colors by coloring each vertex a distinct color. This establishes that the set of lean colorings of $H$, denoted $\mathcal{L}(H)$, is not empty and, hence, the lean number of a hypergraph is bounded above by the size of the vertex set. Let $L$ be a lean coloring of $H$, so that $L\in\mathcal{L}(H)$, and define $|L|$ to be the number of colors used in this lean coloring of $H$. 
Then $\{|L|: L \in \mathcal{L}(H)\}$, i.e., the set consisting of the numbers of colors used among the set of lean colorings of the hypergraph $H$, is a finite subset of the interval of integers $[2,|V|]$. 
By the well-ordering principle, there must be a unique least element in this set, 
which establishes the uniqueness of the lean number of a hypergraph. For example, the hypergraph in \Cref{fig:leancolorings} can be lean colored with as few as $3$ colors and at most $8$ colors, and so its lean number is $3$. 

Having established the existence and uniqueness of the lean number of a hypergraph, we proceed in \Cref{sec:LCARP} to describe the problem of evaluating a hypergraph's lean number in three ways. 
Our work in \Cref{sec:LCARP} allows us to state \Cref{thm:bijection}, which establishes an equality between the following: $1$. The number of connected and reduced hypergraphs on $n$ vertices, $2$. The number of connected and reduced bipartite graphs $(V_1,V_2,E)$, where $|V_1|=n$, and $3$. The number of connected abstract simplicial complexes with $n$ points. 
In \Cref{sec:special hypergraphs}, we prove results on the lean numbers of connected hypergraphs, in particular $k$-uniform, $k$-partite, and $r$-complete hypergraphs. Finally, in \Cref{sec:Algo}, we detail a Python-like algorithm that can produce a bound for the lean number of an arbitrary hypergraph. The explicit algorithm can be found at https://github.com/HarrisonP123/Lean-Number.

Before proceeding to establish the aforementioned results, we give the reader our motivation for this problem, which stems from coloring arcs of knots.

\subsection{Motivation to Study Lean Colorings}
To state our motivation and inspiration for lean colorings, we begin with a quick review of some knot theory concepts. 
A mathematical knot is an embedding of the sphere $S^1$ into $\mathbb{R}^3$. A knot's \textit{projection} is an associated diagram in $\mathbb{R}^2$ depicting the arcs and crossings of the knot when it is ``laid flat."
Accordingly, a mathematical link is an embedding of any finite number of copies of $S^1$ into $\mathbb{R}^3$. 
All that follows applies equally to links, but we restrict our attention to knots. There are properties of knots that do not change under ambient isotopies, which are called \textit{invariants} of the knot. 
An elementary knot invariant is \textit{tricolorability}, and we recall that a knot is \textit{tricolorable} if and only if each arc can be assigned one of three colors so that at each crossing, either the arcs display exactly three distinct colors or exactly one color. \Cref{fig:tricolorornot} includes an example of a knot that is tricolorable and a knot that is not.

\begin{figure}[h]
    \centering
    \includegraphics[width=0.5\linewidth]{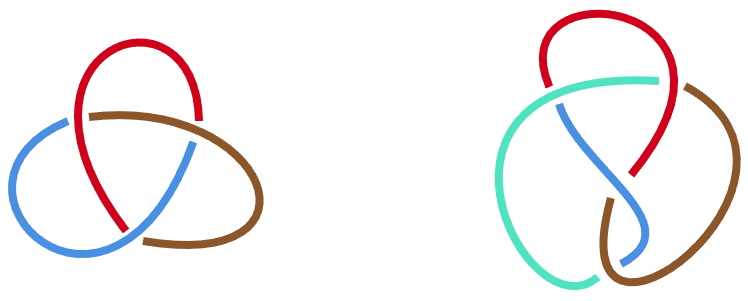}
    \caption{The trefoil knot on the left can be tricolored, but no matter how we try, we cannot tricolor the figure-eight knot on the right. Any coloring following the tricoloring rules of a figure-eight knot requires at least four colors. This is true in any projection of this knot.}
    \label{fig:tricolorornot}
\end{figure}

Categorizing knots by tricolorability is rather limited, since it only allows us to distinguish between two categories of knots-- those that are tricolorable and those that are not. 
Instead, we introduce a knot's \textit{lean~number}, which is the fewest number of colors necessary to color the arcs of a knot such that at each crossing, the arcs display exactly three distinct colors or display exactly one color.
We also require that the lean number of any knot, except the unknot, be at least two. 
For example, since the trefoil knot is tricolorable, its lean number is $3$, as is the case for every tricolorable knot. 
Since the figure-eight knot is not tricolorable, but it can be colored using our rules by four colors, it has lean number $4$. 
It can be verified that the lean number of a knot is invariant under ambient isotopy by showing that it is invariant under Reidemeister moves. 
A knot's lean number allows us to classify knots based on their lean number. This improved tool allows us to distinguish between knots and links much more finely.

When calculating the lean number of knots, it is natural to consider the knot as a collection of arcs and crossings, rather than in its original form. 
By considering each arc of a knot as a vertex and each crossing as an edge (made up of a set of vertices) we derive a natural injection between the set of knot projections (up to Reidemeister moves) and the set of hypergraphs with finitely many vertices. \Cref{fig:knottohyp} illustrates one example of this injection. 
While our main object is to study the lean number of hypergraphs, we return to knots and links in \Cref{sec:LNOSKAL}, where we provide the lean number for $130$ knots and links.

\begin{figure}[h]
    \centering
\includegraphics[width=0.75\linewidth]{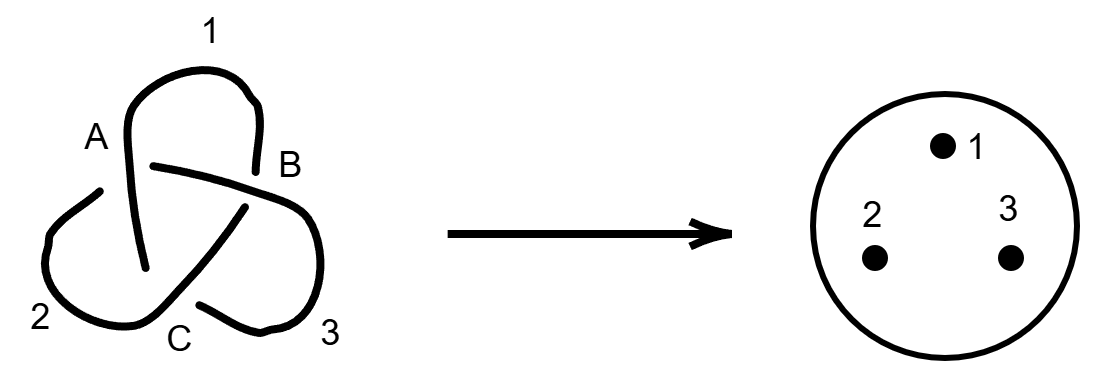}
    \caption{
    Arcs $1,2$, and $3$ are present in the crossing A of the trefoil knot, and so we draw an edge containing vertices $1,2$, and $3$ in the resulting hypergraph.
    The same arcs are present in the crossings labeled B and C, and so the other edges in the resulting hypergraph are $\{2,3,1\}$ and $\{3,1,2\}$, respectively. 
    These edges are considered identical to the first edge.}
    \label{fig:knottohyp}
\end{figure}

\section{Related Coloring Problems}\label{sec:LCARP}
Throughout this article, we make heavy use of the fact that a hypergraph and its reduction share many desirable properties. These are outlined in \Cref{lem:reduction} and \Cref{lem:HypAndItsRed} in \Cref{sec:LCOH}. 
In \Cref{sec:LCOBG}  and \Cref{sec:LCOASC}, we demonstrate that the problem of lean coloring a hypergraph can be described by an analogous coloring of certain classes of bipartite graphs, or equivalently, by certain classes of abstract simplicial complexes. 
\Cref{sec:LCOASC} concludes with \Cref{thm:bijection}, where we establish a bijection between each of the objects discussed in \Cref{sec:LCARP}.

\subsection{The Reduction of a Hypergraph}\label{sec:LCOH}

We begin with some definitions.

\begin{definition}\label{def:reduction}
    A \textit{reduced} hypergraph is one in which no edge is contained in another \cite{Reduction}. 
The \textit{reduction} of a hypergraph $H$ is defined to be $H^- = (V,E^-)$, where 
\[E^- = E \setminus \{e \in E : e \subseteq f \text{ for some } f \in E\setminus\{e\}\}.\]
\end{definition}

Less formally, the reduction of a hypergraph is the hypergraph obtained by removing all edges that are included in another edge, while keeping the vertices as they are. 
\Cref{fig:reduction} displays a hypergraph and its reduction. 
\Cref{def:reduction} speaks of a reduction as unique. \Cref{lem:reduction} demonstrates this.

\begin{figure}[h]
    \centering
    \begin{subfigure}[b]{0.3\textwidth}
        \includegraphics[width=\textwidth]{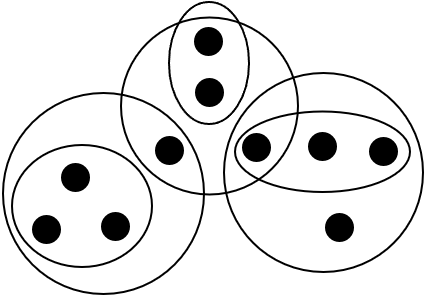}
        \caption{Hypergraph $H$}
        \label{fig:color1}
    \end{subfigure}
    ~
    \begin{subfigure}[b]{0.3\textwidth}
        \includegraphics[width=\textwidth]{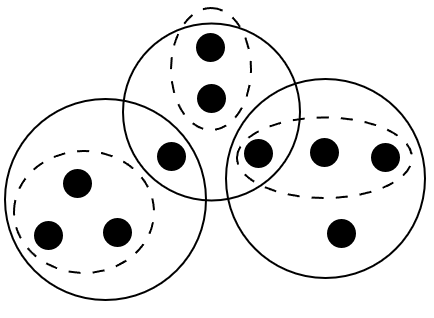}
        \caption{$E'$ drawn with dashed lines}
        \label{fig:color2}
    \end{subfigure}
    ~
    \begin{subfigure}[b]{0.3\textwidth}
        \includegraphics[width=\textwidth]{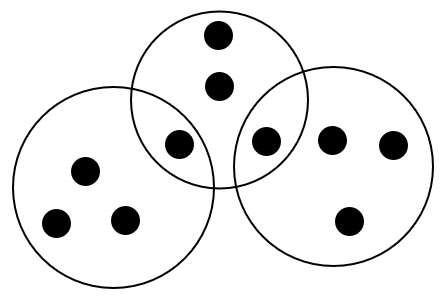}
        \caption{Reduced Hypergraph $H^-$}
        \label{fig:color3}
    \end{subfigure}
    \caption{The hypergraph  in subfigure (A) has six edges, three of which are proper subsets of other edges. The edges to be removed, i.e., $E'=\{e \in E : e \subseteq f \text{ for some } f \in E\setminus\{e\}\}$, are drawn with dashed lines in subfigure (B). The reduced hypergraph $H^-$, displayed in subfigure (C), has the dashed edges removed but keeps each edge in $E$ that is not in $E'$. Note that both the hypergraph $H$ and the reduced hypergraph $H^-$ are connected.}\label{fig:reduction}
\end{figure}

\begin{lemma}
\label{lem:reduction}
Given a hypergraph $H = (V,E)$, its reduction $H^-$ is unique. Additionally, $H$ is connected if and only if $H^-$ is connected.
\end{lemma} 
\begin{proof}
For the first statement, suppose $H_1 = (V,E_1)$ and $H_2 = (V,E_2)$ are reductions of $H = (V,E)$. Then $E_1,E_2 \subseteq E$. If $E_1 \neq E_2$, then without loss of generality, there exists an $e \in E_1$ such that $e \notin E_2$. Since $e \notin E_2$, there must be some $f \in E$ distinct from $e$ such that $e \subseteq f$. But then $e$ would not be in $E_1$, which is a contradiction. 

For the second statement, in the forward direction, suppose $H$ is connected, i.e., for every pair $v_1,v_n \in V$, there exists a path $P: v_1, e_1,v_2,\ldots,v_{n-1},e_{n-1},v_n$ such that each $e_k$ contains the vertices $v_k$ and $v_{k+1}$ for $1 \leq k \leq n-1$. We modify the path $P$ to a new path $P'$ in the following way: for each $e_k$, substitute in any $e_k' \in E$ that contains $e_k$ and is maximal with respect to inclusion. Then $P'$ is still a path from $v_1$ to $v_n$, since each edge $e_k'$ in $P'$ still contains $v_k$ and $v_{k+1}$ for $1 \leq k \leq n-1$. Further, $P'$ is a path from $v_1$ to $v_n$ in the reduction $H^-$, since it only contains edges that are kept when $H$ is reduced. 
It follows that $H^-$ is connected.

For the converse, suppose the reduction $H^-$ is connected, i.e., for every pair of vertices $v_1,v_n \in V$, there exists a path $P$ between them in $H^-$. After we add in edges to $H^-$ to reconstruct $H$, the same path $P$ can be used to connect $v_1$ and $v_n$ in $H$. It follows that if $H^-$ was connected, then so is $H$. 
\end{proof}

Our next result shows that $H$ and its reduction $H^-$ have the same lean number.

\begin{lemma} \label{lem:HypAndItsRed}
The lean colorings of the hypergraph $H=(V,E)$ are exactly the lean colorings of its reduction $H^-=(V,E^-)$. Consequently, $\Lean(H)=\Lean(H^-)$.
\end{lemma}

\begin{proof}
    Let $C_1$ be a lean coloring of $H$ and let $f \in E$. If $f$ does not contain any other edges, then reducing $H$ does not change the coloring of the vertices in the edge $f$ in $H^-$. 
    Suppose $f\in E^-$ such that in $H$, the edge $f$ properly contains another edge, $e\neq f$. 
    Under $C_1$, we have used at least two colors and the edge $f$ is either monocolored or multicolored.
    After the removal of edge $e$, we have still used at least two colors and $f$ is still either multicolored or monocolored. 
    It follows that $C_1$ is a lean coloring of $H^-$.
    
    Conversely, let $C_2$ be a lean coloring of $H^-$ and let $f \in E^-$. Under $C_2$, we have used at least two colors and $f$ is either multicolored or monocolored. 
    After adding any number of subedges to $f$ in order to match the edge set of $H$, the coloring $C_2$ uses at least two colors. Moreover, each edge contained in $f$ is either multicolored or monocolored. It follows that $C_2$ is a lean coloring of $H$. 
\end{proof}

With these results at hand, we now focus on an equivalent coloring problem on bipartite graphs.

\subsection{The Bilean Coloring of Bipartite Graphs}\label{sec:LCOBG}
We now consider the lean coloring problem in the context of certain colorings of bipartite graphs. 
Recall that a bipartite graph is a graph $G = (V,E)$ with vertex set $V$ and edge set $E$, with the additional condition that $V$ can be partitioned into two subsets, $V_1$ and $V_2$, so that each edge in $E$ is incident to one vertex in the left vertex set $V_1$ and one vertex in the right vertex set $V_2$. We denote such a bipartite graph as $G = (V_1,V_2,E)$. The partition of $V$ into sets $V_1$ and $V_2$ is called a \textit{bipartition of $G$}.

Given a hypergraph $H = (V,E)$, we construct $G_H=(V,E,E')$, an associated bipartite graph with vertex set $V\cup E$ and edge set $E'$ called the \textit{incidence graph of $H$} or \textit{Levi graph of $H$}. The bipartite graph $G_H$ is constructed in the following way: 
the set of left vertices of $G_H$ is $V$, the set of right vertices of $G_H$ is $E$, and if $v\in V$ belongs to the edge $e \in E$ in $H$, then $(v,e) \in E'$. 
See \Cref{fig:HypToBip} for an example. 

\begin{remark}\label{rem:bipartite to hypergraph}
Recall that all of our hypergraphs have no empty edges. Hence, we may apply the result of Bahmanian and Sajna \cite[Theorem~3.11]{connectedhyps} which states: A hypergraph $H = (V,E)$ is connected with no empty edges exactly when its incidence graph $G_H$ is connected.
This can also be done in reverse:
given a bipartite graph $G = (V,E)$, with a specified bipartition, there exists a unique hypergraph $H$ such that $G_H=G$, i.e., we may treat $G$ as the incidence graph of a hypergraph $H$. 
This establishes a correspondence between hypergraphs (with no empty edges) and bipartite graphs with specified bipartition.
\end{remark}

\begin{figure}[h]
    \begin{subfigure}[b]{0.5\textwidth}
        \centering
        \includegraphics[width=0.75\textwidth]{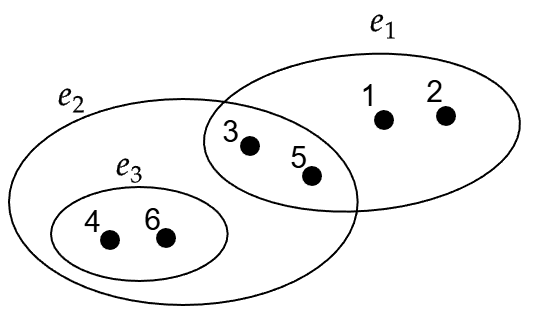}
        \caption{A hypergraph $H$}
        \label{fig:HypToBipA}
    \end{subfigure}
    ~
    \begin{subfigure}[b]{0.5\textwidth}
        \centering 
        \includegraphics[width=0.5\textwidth]{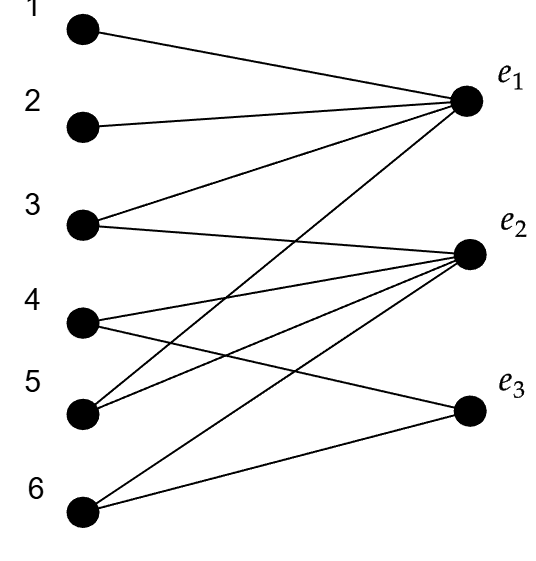}
        \caption{The incidence graph $G_H=(V,E,E')$}
        \label{fig:HypToBipB}
    \end{subfigure}
    \caption{Subfigure (A) is a hypergraph $H=(V,E)=(\{1,2,3,4,5,6\},\{e_1,e_2,e_3\})$.
    Subfigure (B) is the corresponding incidence graph $G_H$, which has six vertices (again denoted $1$ through $6$) in its left vertex set $V$ and three vertices in its right vertex set $E$ (denoted $e_1,e_2,e_3$). Note that in $G_H=(V,E,E')$ there is an edge $e'\in E'$ between a vertex $v\in V$ and a vertex $e\in E$ if $v\in e$ in $H$.}
    \label{fig:HypToBip}
\end{figure}

Given the correspondence between a hypergraph $H$ and its incidence graph $G_H$, we now define \textit{bilean coloring}, which is an analogue of lean coloring defined for bipartite graphs. 
Given a bipartite graph $G~=~(V_1,V_2,E)$, we define a \textit{bilean coloring of $G$} to be an assignment of \underline{at least two} colors to the vertices in $V_1$ so that for each vertex $v \in V_2$, the vertices in $V_1$ adjacent to $v$ display either all distinct colors or exactly one color. We say that the set of all bilean colorings of a bipartite graph $G$ is $\mathcal{B}(G)$. 
The \textit{bilean number of $G$}, denoted $\Bilean(G)$, is the fewest number of colors necessary to bilean color the bipartite graph $G$. \Cref{fig:colored bips} includes bilean colorings of two bipartite graphs.

\begin{figure}[h]
    \begin{subfigure}[b]{0.5\textwidth}
        \centering
        \includegraphics[width=0.25\textwidth]{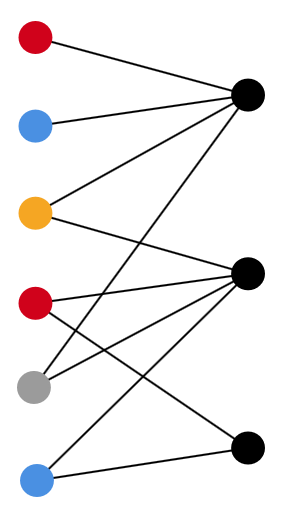}
        \caption{One bilean colored bipartite graph.}
        \label{fig:bipcolor1}
    \end{subfigure}
    ~
    \begin{subfigure}[b]{0.5\textwidth}
        \centering \includegraphics[width=0.25\textwidth]{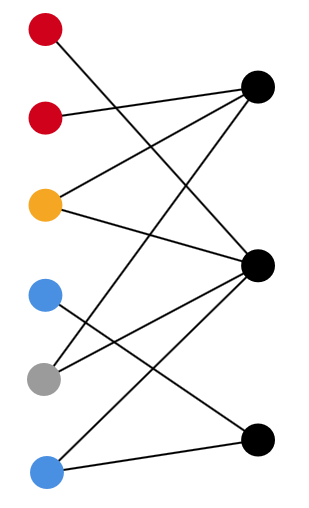}
        \caption{Another bilean colored bipartite graph.}
        \label{fig:bipcolor2}
    \end{subfigure}
    \caption{An illustration of two bilean colored bipartite graphs. The vertices adjacent to each black vertex present either all distinct colors or exactly one color. Note each bilean coloring uses at least two colors, as minimally required.}
    \label{fig:colored bips}
\end{figure}

Given \Cref{lem:HypAndItsRed}, it is preferable and sufficient to work with reduced hypergraphs. Hence, we define a property which for bipartite graphs is analogous to being a reduced hypergraph. 

\begin{definition}\label{def:neighbor}
    Let $G = (V_1,V_2,E)$ be a connected and simple bipartite graph and let $N_G(g)$ be the set of vertices in $V_1$ adjacent to $g \in V_2$. We define $G$ to be \textit{reduced} if, for each distinct pair $g,g' \in V_2$, it is the case that $N_G(g) \not \subseteq N_G(g')$. 
\end{definition}

For example, in \Cref{fig:HypToBip}, the graph in subfigure (B) is not reduced 
since $N_G(e_3)=\{4,6\}\subseteq N_G(e_2)=\{3,4,5,6\}$,
and the graph in \Cref{fig:colored bips} subfigure (B) is reduced as no neighborhood of a right vertex is contained in a neighborhood of another right vertex. 
We use \Cref{def:neighbor} to arrive at  the following result.

\begin{lemma}\label{lem:ConRedImpliesRed}
A hypergraph $H$ is reduced if and only if its incident graph $G_H$ is reduced.
\end{lemma}

\begin{proof}
    Since $H$ is reduced, there is no distinct pair of edges $e,e'\in E$ such that $e \subseteq e'$. 
    In the incidence graph, $G_H = (V,E,E')$, this means that there are no vertices $e,e' \in E$ such that $N_{G_H}(e) \subseteq N_{G_H}(e')$. It follows that $G_H$ is reduced.
    For the converse, suppose $H$ is not reduced, that is, there exist a pair of distinct and nonempty edges $e,e'\in E$ such that $e \subseteq e'$. It would then be the case that $N_{G_H}(e) \subseteq N_{G_H}(e')$, and so $G_H$ is not reduced.
\end{proof}

In the same way that determining the lean number of a hypergraph is deeper when we restrict our attention to connected and reduced hypergraphs, our exploration of bilean numbers is deeper when we restrict our attention to connected and reduced bipartite graphs. 
Henceforth, we say a hypergraph or a bipartite graph is {CR} if it is both connected and reduced.
We recall that the number of non-isomorphic connected bipartite graphs that are labeled (resp.\ unlabeled) on a given number of vertices is known, for their enumeration \newline see \cite[OEIS \seqnum{A001832}]{OEISConnLab} (resp.\ \cite[OEIS \seqnum{A005142}]{OEISConnUnlab}). 
The number of CR
bipartite graphs on $n$ vertices is not obvious. 
What we do know, from \Cref{lem:ConRedImpliesRed} and \Cref{rem:bipartite to hypergraph}, is that the set of
CR bipartite graphs with $n$ vertices in the left set 
is in bijection with the set of CR hypergraphs on $n$ vertices.
In the next section, we establish a new bijection between CR hypergraphs on $n$ vertices and the number of connected simplicial complexes on $n$ vertices. 

\subsection{The Silean Coloring of Abstract Simplicial Complexes}\label{sec:LCOASC}
Throughout, $X$ is a finite set.
Recall that an abstract simplicial complex is a pair of sets $(X,S)$, where $S \subseteq \mathcal{P}(X)$, which satisfies the following conditions:
\begin{itemize}
    \item if $T \in S$ and $A \subseteq T$, then $A \in S$, and 
    \item each singleton subset of $X$ is in $S$. 
    \end{itemize}
The elements of $X$ are called \textit{points} of the abstract simplicial complex. By considering $X$ as a set of vertices and $S$ as a set of edges, we may instead consider the abstract simplicial complex $(X,S)$ as a hypergraph $H = (X,S)$. 
For example, \Cref{fig:ASC} displays an abstract simplicial complex on four points along with its corresponding hypergraph. 

\begin{figure}[h]
    \begin{subfigure}[b]{0.5\textwidth}
        \centering
        \includegraphics[width=0.5\textwidth]{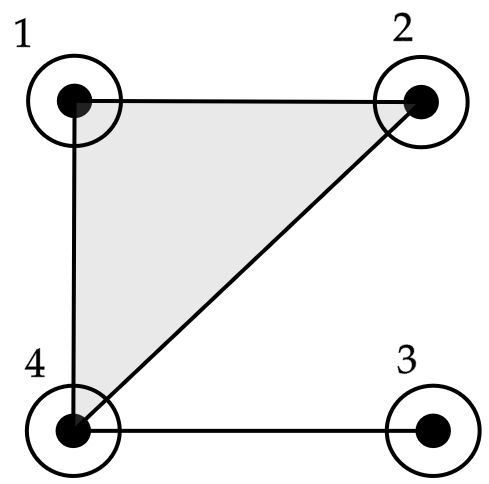}
        \caption{ }
        \label{fig:ASC and Hyp A}
    \end{subfigure}
    ~
    \begin{subfigure}[b]{0.5\textwidth}
        \centering \includegraphics[width=0.5\textwidth]{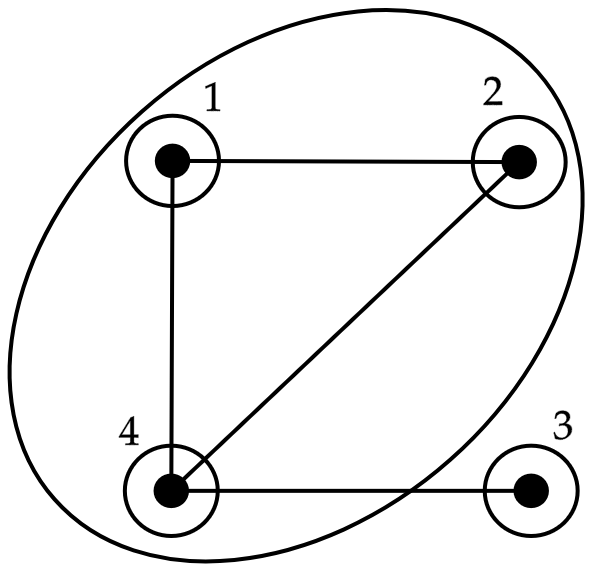}
        \caption{ }
        \label{fig:ASC and Hyp 2}
    \end{subfigure}
    \caption{Subfigure (A) is an illustration of the abstract simplicial complex $(X=\{1,2,3,4\},$ $S=\{\{1\},\{2\},\{3\},\{4\},\{1,2\},\{1,4\},\{2,4\},\{3,4\},\{1,2,4\}\})$. Circles around each point indicate a singleton in $S$. Line segments between pairs of points indicate a two-element set in $S$. The shaded area indicates a three-element set in $S$. Subfigure (B) is an illustration of the corresponding hypergraph $H=(X,S)$.}
    \label{fig:ASC}
\end{figure}

The connectivity of an abstract simplicial complex is strongly tied to the connectivity of its associated hypergraph, as the following result shows.  

\begin{lemma}
\label{lem:ASCconn}
An abstract simplicial complex is connected if and only if its associated hypergraph is connected.
\end{lemma}

\begin{proof}
   A connected abstract simplicial complex $C = (X,S)$ is one where there exists a path joining every pair of distinct $x_1,x_n \in X$, i.e., a sequence of the form $x_1,s_1,x_2,s_2, \ldots, s_{n-1}, x_n$, where each set $s_k \in S$ contains $x_k, x_{k+1} \in X$. 
   When treated as a hypergraph $H = (X,S)$, it follows that there exists a path between each pair of distinct vertices in $H$, and so $H$ is connected. 
   In the reverse direction, let a connected hypergraph $H = (V,E)$ correspond to an abstract simplicial complex $D = (V,E)$. Then there exists a path between each pair of distinct vertices in $H$, and so there exists a path between each pair of distinct vertices in $D$. Hence, $D$ is connected. 
\end{proof}

Each hypergraph has a unique reduction, as we proved in \Cref{lem:reduction}. Also, to each hypergraph, we may introduce each subset of each edge to the edge set to construct a new hypergraph. We define this next.

\begin{definition}
    \label{def:accession}
    Let $H = (V,E)$ be a hypergraph. The \textit{accession of $H$} is $H^+ = (V,E^+)$, where \[E^+ =\left( \bigcup_{e \in E} \mathcal{P}(e)\right) \cup\{\{v\}:v\in V\mbox{ is an isolated vertex}\}.\] 
\end{definition}

\Cref{fig:accessionexample} shows an example of a hypergraph and its accession.

\begin{figure}[h!]
    \centering
    \begin{subfigure}[b]{0.5\textwidth}
        \centering
        \includegraphics[width=0.5\textwidth]{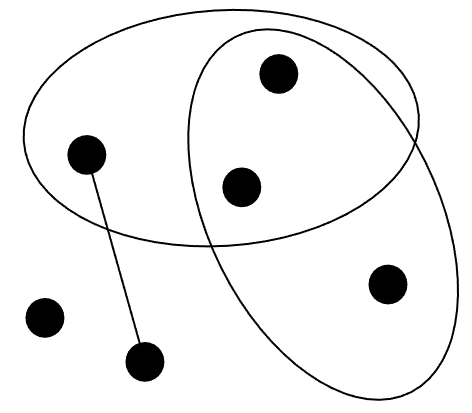}
        \caption{}
        \label{fig:HypToAccessA}
    \end{subfigure}
    ~
    \begin{subfigure}[b]{0.5\textwidth}
        \centering \includegraphics[width=0.5\textwidth]{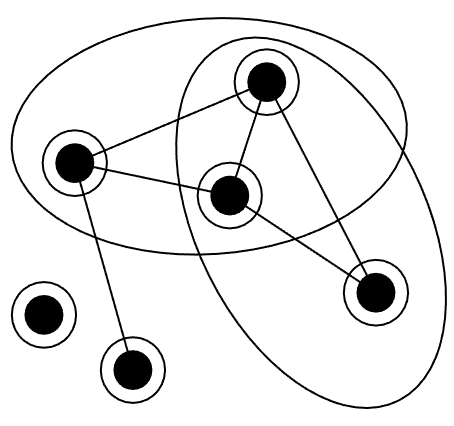}
        \caption{}
        \label{fig:HypToAcess2}
    \end{subfigure}
    \caption{In subfigure (A) we illustrate a hypergraph and we use a straight line between two vertices to signify an edge of size two.
    The accession of the hypergraph $H$ in subfigure (A) is the hypergraph $H^+$ in subfigure (B). 
    Each subset of each edge of $H$ is included in the edge set of the accession $H^+$.}
    \label{fig:accessionexample}
\end{figure}

As stated, \Cref{def:accession} refers to each accession as unique. Next, we prove this is indeed the case.

\begin{lemma}
    \label{lem:!access}
    For a given hypergraph $H$, its accession $H^+$ is unique. Additionally, $H$ is connected if and only if $H^+$ is connected.
\end{lemma}
\begin{proof}
    We argue the first statement by contradiction. Suppose $H_1^+=(V,E_1)$ and $H_2^+ = (V,E_2)$ are both accessions of $H = (V,E)$, with $E_1 \neq E_2$. 
    Then without loss of generality, there exists an $e \in E_1^+$ such that $e \notin E_2^+$. 
    This implies that either $e \in E$ (which cannot be the case. since then $e \in E_2^+$) or $e$ is a strict subset of an edge $f \in E$. But then $e$ being a strict subset of $f \in E$ would imply $f$ and its subsets are edges in $E_2^+$, which contradicts $e$ not being in $E_2^+$.
    
    Let us now prove the second statement. Note first that the reduction of $H$, $H^-$, is identically the reduction of $H^+$, i.e., $(H^+)^-$. This is the case since the edges of $H$ that are not in $H^-$ and the edges of $H^+$ that are not in $H^-$ are subsets of edges of $H^-$. Then \Cref{lem:reduction} implies $H$ is connected if and only if $H^-$ is connected, which occurs if and only if $H^+$ is connected.
\end{proof}

This notion of accession is constructed in such a way to ensure that a given hypergraph can become \textit{downward-closed}, i.e., a hypergraph in which each subset of each edge is also an edge. 
An alternative way to represent abstract simplicial complexes is via downward-closed hypergraphs. 
This connection allows us to present a proof that CR hypergraphs are in bijection with connected abstract simplicial complexes.

\begin{lemma}
    \label{lem:hypsequaltoASC}
    Let $n \in \mathbb{N}=\{1,2,3,\ldots \}$. The number of CR hypergraphs on $n$ vertices is equal to the number of connected abstract simplicial complexes with $n$ points.
\end{lemma}

\begin{proof}
    Let $H = (V,E)$ be a CR hypergraph and $|V| = n$. Its accession $H^+ = (V,E^+)$ satisfies the following property: if $e \in E^+$ and $f \subseteq e$, then $f \in E^+$. Additionally, for each $v \in V$, the singleton $\{v\}$ is in $E^+$ (since $H$ is connected). Also $H^+$ is uniquely determined, as shown in \Cref{lem:!access}. It follows that the pair of sets $(V,E^+)$ satisfies the definition of an abstract simplicial complex, and is connected by \Cref{lem:ASCconn}. 
    
    For the converse, let $C = (X,S)$ be a connected abstract simplicial complex. By its definition, $S$ contains all singleton subsets of $X$ and is closed under taking subsets. We may consider $C$ as a hypergraph with $X$ as the vertex set and $S$ as the edge set. The reduction of $C$ is then a hypergraph that is connected by \Cref{lem:ASCconn} and uniquely determined by \Cref{lem:reduction}. 
\end{proof}

In \Cref{fig:simplex to hyp}, we present each connected abstract simplicial complex on three points and give their associated hypergraphs.

\begin{figure}[h!]
    \centering
    \begin{tabular}{rrr}
    \begin{subfigure}[b]{0.3\textwidth}
        \centering
        \includegraphics[width=0.75\textwidth]{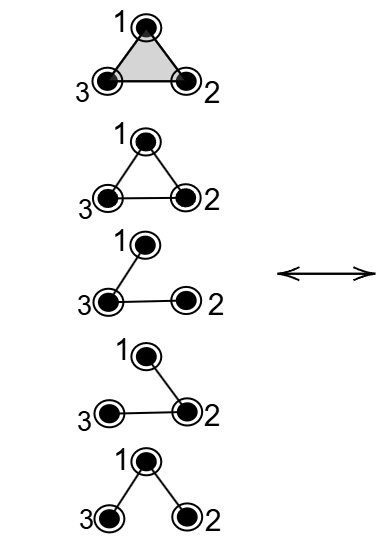}
        \caption{Abstract simplicial complexes on three points}
        \label{fig:bipcolor1}
    \end{subfigure}
    &
    \begin{subfigure}[b]{0.3\textwidth}
    \centering
\includegraphics[width=0.75\textwidth]{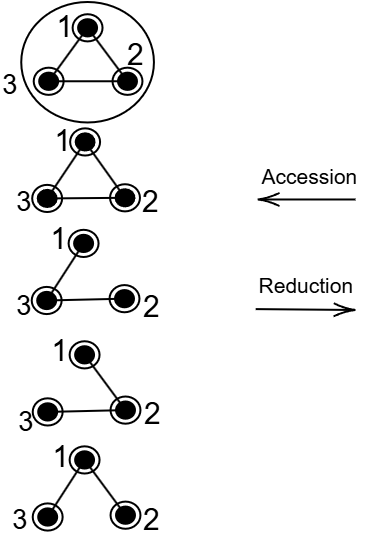}
        \caption{Abstract simplicial complexes considered as hypergraphs}
        \label{fig:bipcolor2}
    \end{subfigure}
    &
    \begin{subfigure}[b]{0.3\textwidth}
    \centering\includegraphics[width=0.3\textwidth]{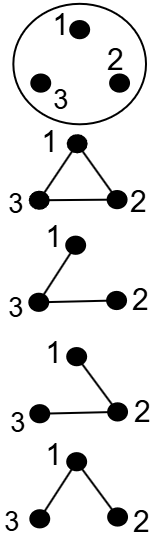}
        \caption{Associated\\ hypergraphs}
        \label{fig:bipcolor2}
    \end{subfigure}
    
    \end{tabular}
    \caption{In subfigure (A), we illustrate all of the distinct connected abstract simplicial complexes on three points. Each circle around a point indicates a singleton set, each line segment indicates a set of two points, and the shaded area represents a set of three points. In subfigure (B), we illustrate the associated hypergraph for each corresponding abstract simplicial complex in subfigure (A). In subfigure (B), the circles around each point represent an edge containing exactly one vertex, the line segments indicate an edge with two vertices, and the large circle represents an edge containing three vertices. 
    In subfigure (C), we reduce each hypergraph in subfigure (B). Accession allows us to produce a unique abstract simplicial complex for each connected reduced hypergraph.}
    \label{fig:simplex to hyp}
\end{figure}

\newpage The following result collects our main findings so far.

\begin{theorem}\label{thm:bijection}
    Let $n \in \mathbb{N}$. The following are equal:
    \begin{enumerate}
        \item The number of connected and reduced hypergraphs on $n$ vertices.
        \item The number of connected and reduced bipartite graphs $(V_1,V_2,E)$, where $|V_1|=n$.
        \item The number of connected abstract simplicial complexes with $n$ points. 
    \end{enumerate} 
\end{theorem}

\Cref{thm:bijection} allows us to take advantage of the fact that the number of connected abstract simplicial complexes is known and appears as OEIS sequence \cite[\href{https://oeis.org/A048143}{A048143}]{OEISConnSimp}, and therefore we also know the size of the other sets. Since each reduced hypergraph (resp. reduced bipartite graphs of the above form, abstract simplicial complex) is a union of CR hypergraphs (resp. CR bipartite graphs of the above form, connected abstract simplicial complexes), we have the following corollary:

\begin{corollary}\label{cor:bijection}
    Let $n \in \mathbb{N}$. The following are equal:
    \begin{enumerate}
        \item The number of reduced hypergraphs on $n$ vertices.
        \item The number of reduced bipartite graphs $(V_1,V_2,E)$, where $|V_1|=n$.
        \item The number of abstract simplicial complexes with $n$ points. 
    \end{enumerate} 
\end{corollary}

Further, the above numbers appear in the OEIS sequence \cite[\href{https://oeis.org/A307249}{A307249}]{OEISConnSimp}. We now consider the lean coloring problem in the context of abstract simplicial complexes. 
To this end, we introduce the following definition.

\begin{definition}
For an abstract simplicial complex $C = (X,S)$,
a \textit{silean coloring} is an assignment of \underline{at least two} colors to the points in $X$ such that the points in each maximal set in $S$ (with respect to inclusion) display either all distinct colors or exactly one color. We say that the set of all silean colorings of an abstract simplicial complex $C$ is $\mathcal{S}(C)$.
The fewest colors necessary to silean color an abstract simplicial complex $C$ is its \textit{silean number}, which is denoted by $\Silean(C)$.   
\end{definition}

In \Cref{fig:ColoredComplex}, we illustrate two silean colorings of an abstract simplicial complex. 

\begin{figure}[h]
    \centering
    \begin{tabular}{rr}
    \begin{subfigure}[b]{0.3\textwidth}
        \centering
        \includegraphics[width=0.75\textwidth]{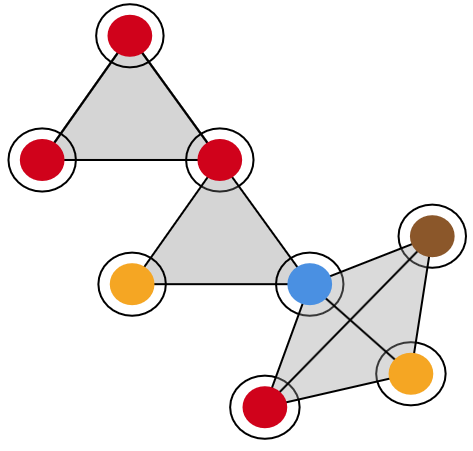}
        \caption{}
        \label{fig:ColoredComplexA}
    \end{subfigure}
    &
    \begin{subfigure}[b]{0.3\textwidth}
    \centering
\includegraphics[width=0.75\textwidth]{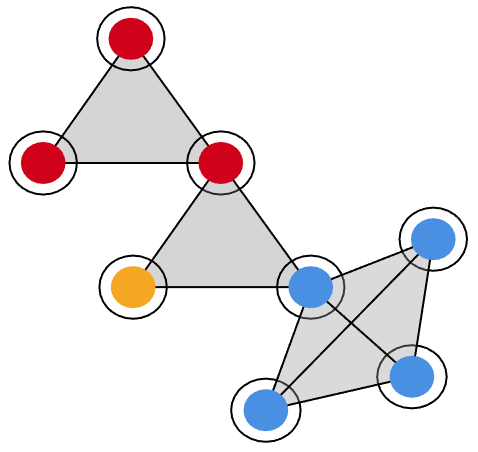}
        \caption{}
        \label{fig:ColoredComplexB}
    \end{subfigure}
    \end{tabular}
    \caption{(A) A silean coloring of a simplicial complex using $4$ colors, and one (B) using 3 colors. The silean number of this abstract simplicial complex is $3$.}
    \label{fig:ColoredComplex}
\end{figure}

Given an abstract simplicial complex $C$, it would be desirable for its silean number to be equal to the lean number of its associated hypergraph, and in turn, have that be equal to the bilean number of that hypergraph's incidence graph. We next give \Cref{thm:LSB}, which relates lean, silean, and bilean numbers.

\begin{theorem}\label{thm:LSB}
    Let $H = (V,E)$ be a reduced hypergraph and let $C = (V,E^+)$ and $G = (V,E,E')$ be its associated abstract simplicial complex and its associated bipartite incidence graph, respectively. Then $\Lean(H)=\Silean(C)=\Bilean(G)$.
\end{theorem}

\begin{proof}
    In the case that one of $H$, $C$, or $G$ (and therefore all three) are disconnected, then $\Lean(H) = \Silean(C) = \Bilean(G) = 2$. 
    Note, the silean colorings of $C$, $\mathcal{S}(C)$, are exactly the lean colorings of $C$ when it is treated as a hypergraph. 
    When we treat $C$ as a hypergraph, then it is the case that $H$ is the reduction of $C$, so \Cref{lem:HypAndItsRed} implies that the $\mathcal{S}(C)$ is exactly $\mathcal{L}(H)$. It follows that $\Silean(C) = \Lean(H)$. Moreover, the set of lean colorings of $H$, $\mathcal{L}(H)$, is exactly the set of bilean colorings of $G$, i.e., $\mathcal{B}(G)$, and so $\Bilean(G) = \Lean(H)$.
\end{proof}

\begin{remark}
    In this section, we have established that computing lean numbers can be approached in three distinct ways (\Cref{thm:LSB}). See \Cref{fig:LSB} for an example of how a lean coloring can be translated into silean and bilean colorings. Based on those equivalent perspectives, this allows us to consider using other combinatorial tools to compute the lean number of a hypergraph. 
    We proceed to consider the problem for certain families of hypergraphs.
\end{remark}

\begin{figure}[h]
    \centering
    \begin{tabular}{rr}
    \begin{subfigure}[b]{0.3\textwidth}
        \centering
        \includegraphics[width=0.75\textwidth]{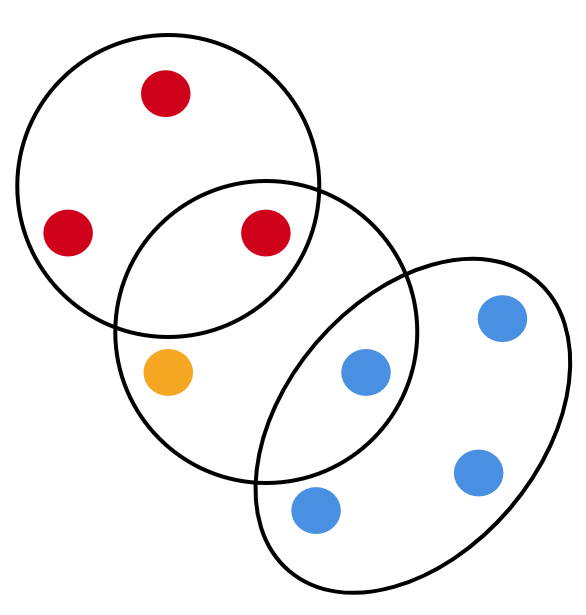}
        \caption{}
        \label{fig:LSB1}
    \end{subfigure}
    &
    \begin{subfigure}[b]{0.3\textwidth}
    \centering
\includegraphics[width=0.75\textwidth]{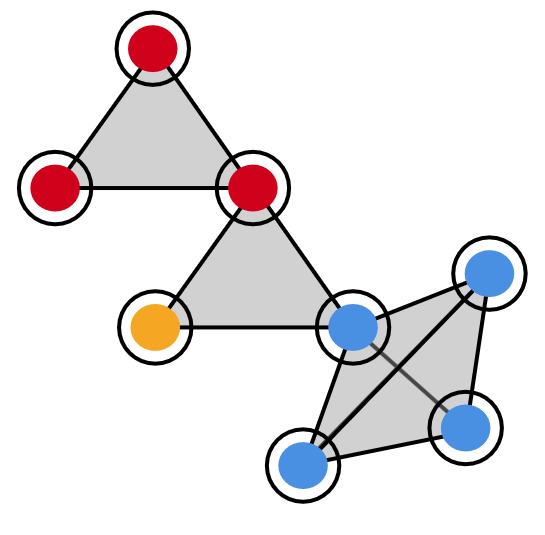}
        \caption{}
        \label{fig:LSB2}
    \end{subfigure}

    \begin{subfigure}[b]{0.3\textwidth}
    \centering
\includegraphics[width=0.75\textwidth]{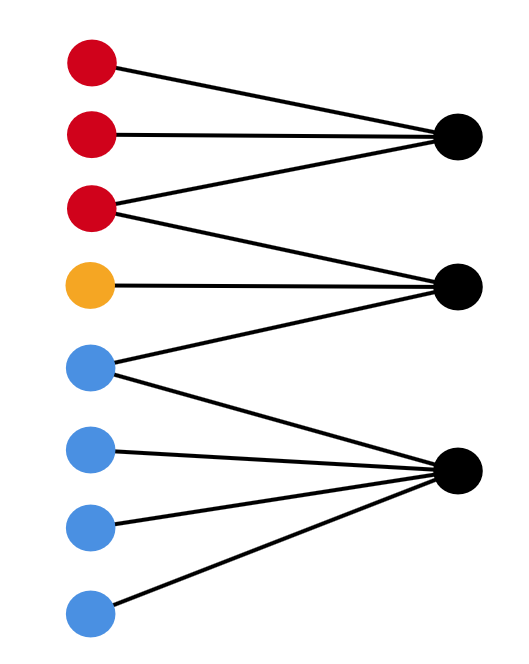}
        \caption{}
        \label{fig:LSB3}
    \end{subfigure}
    \end{tabular}
    \caption{(A) An optimal lean coloring of a CR hypergraph $H$. (B) An optimal silean coloring of the abstract simplicial complex corresponding to $H$. (C) An optimal bilean coloring of the incidence graph of $H$.}
    \label{fig:LSB}
\end{figure}

\section{Special Connected Hypergraphs}
\label{sec:special hypergraphs}
If a hypergraph is disconnected, we may simply color the vertices of one component one color and color the vertices in the remaining component(s) another color, quickly demonstrating a lean number of $2$. Connectedness of a hypergraph, then, is an essential assumption to ensure more interesting results. In this section, we assume connectedness in addition to other conditions on a hypergraph in order to place bounds on its lean number. \Cref{lem:EdgesAtLeastK} and \Cref{cor:unif} concern hypergraphs with edges that each contain at least a fixed number of vertices. \Cref{cor:kpar} concerns hypergraphs that are $k$-partite.

\begin{lemma}\label{lem:EdgesAtLeastK}
A connected hypergraph $H = (V,E)$ whose edges each have size at least $k$ satisfies 
\[\Lean(H)\geq k.\]
\end{lemma}

\begin{proof}
For $k=1$ and $k=2$, the conclusion holds, as the lean number of any hypergraph is at least 2. 
Hence, we assume $k\geq 3$. If there exists a lean coloring $C$ of $H$ with $k-1$ or fewer colors, then under $C$, each edge in $E$ must be monocolored. 
We partition the edge set as follows. We fix a color, calling it color 1, and we let $E_1$ contain all edges in $E$ that, under the coloring $C$, have their vertices colored with color 1. Then, let $E_2=E\setminus E_1$, i.e., $E_2$ is the set of edges in $E$ that, under the coloring $C$, have their vertices not colored with color 1. 
The sets $E_1$ and $E_2$ form a partition of the edge set (and share no vertices). It follows that the subhypergraphs induced by the edge sets $E_1$ and $E_2$ are disconnected, violating our assumption of connectedness. Thus, we can conclude that any coloring of $H$ must use at least $k$ colors.
\end{proof}

The previous result applies also to $k$-uniform hypergraphs, whose edges each contain exactly $k$ vertices.

\begin{corollary}
\label{cor:unif}
A connected $k$-uniform hypergraph $H$ satisfies $\Lean(H) \geq k$.
\end{corollary}

Let $k \geq 2$. A hypergraph $H$ is $k$-\textit{partite} if the vertices can be partitioned into $k$ sets, i.e., $V_1, V_2, \ldots, V_k$, and each edge contains exactly one vertex from each $V_i$ for $1 \leq i \leq k$. 
The following result concerns the lean number of such hypergraphs.

\begin{corollary}\label{cor:kpar}
If $k\geq 2$, then a connected $k$-partite hypergraph $H$ satisfies $\Lean(H) = k$. 
\end{corollary}

\begin{proof}
    Such a hypergraph has lean number at most $k$, since we may color each vertex in $V_i$ with color $i$, and then each edge is multicolored. A $k$-partite hypergraph is also $k$-uniform, so by \Cref{cor:unif}, it has a lean number at least $k$. 
\end{proof}

\begin{remark}
    \Cref{cor:kpar} is a tidy parallel to the fact that bipartite graphs are exactly the graphs with chromatic number $2$. It is \textit{not} the case that a hypergraph with lean number $k$ is necessarily $k$-partite, though, as \Cref{cor:rcomp} demonstrates below. 
\end{remark}

\section{General Results}\label{sec:GeneralResults}
We begin by introducing wide-path connected hypergraphs in \Cref{def:wide path}, and as a special case of these, we give the lean number of $r$-complete hypergraphs, see \Cref{cor:rcomp}.
We then consider hypergraphs in greater generality. In \Cref{lem:separation}, we introduce a way to bound the lean number of an arbitrary hypergraph in terms of the lean numbers of certain subhypergraphs. We then show in \Cref{thm:GoldenGoose} that the lean number of a hypergraph is preserved in some components after cleaving. Additionally, we show in \Cref{thm:DiamondGoose} that the order in which we perform cleaving operations on a hypergraph is inconsequential. 
We then relate the lean number for any hypergraph to the
chromatic number of its $2$-section, see \Cref{lem:2Section}.

To begin we introduce some needed concepts. For a hypergraph $H = (V,E)$, its \textit{$2$-section}, denoted $[H]_2$, is a graph $G= (V,E')$ such that if $v_1, v_2 \in V$ are adjacent in $H$, then $(v_1,v_2) \in E'$. 
We recall that a coloring of the vertices of $H$ requiring that adjacent vertices are assigned different colors is called a \textit{strong coloring} of $H$, for more on strong colorings we refer the reader to \cite{Hypergraphs}.
The fewest number of colors necessary to strong color a hypergraph $H$ is called its \textit{strong chromatic number}. Recall also that the chromatic number of a graph $G$, which we denote $\chi(G)$, is the smallest number of colors needed to color its vertices so that no two adjacent vertices share the same color. Finally, we define a \textit{wide-path connected} hypergraph.

\begin{definition}\label{def:wide path}
    Let $H= (V,E)$ be a hypergraph with $|V|\geq 2$.
    Then $H= (V,E)$ is \textit{wide-path connected} if for each pair of edges $e_1,e_n \in E$, there exists a sequence of edges $e_2,e_3,\ldots,e_{n-1} \in E$ such that for each $1 \leq i \leq n-1$, $|e_i \cap e_{i+1}| \geq 2$. The sequence $P: e_1,e_2,\ldots,e_n$ is referred to as a \textit{wide path between $e_1$ and $e_n$}.
\end{definition}

\Cref{def:wide path} implies each wide-path connected hypergraph is also connected. \Cref{fig:WidePath} shows a hypergraph that is wide-path connected and one that is not.

\begin{figure}[h!]
    \centering
    \begin{tabular}{rr}
    \begin{subfigure}[b]{0.5\textwidth}
        \includegraphics[width=0.75\textwidth]{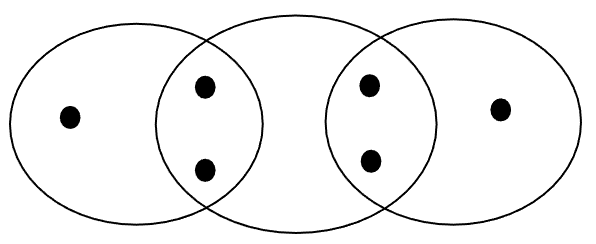}
        \caption{A wide-path connected hypergraph}
    \end{subfigure}
    &
    \begin{subfigure}[b]{0.5\textwidth}
\includegraphics[width=0.75\textwidth]{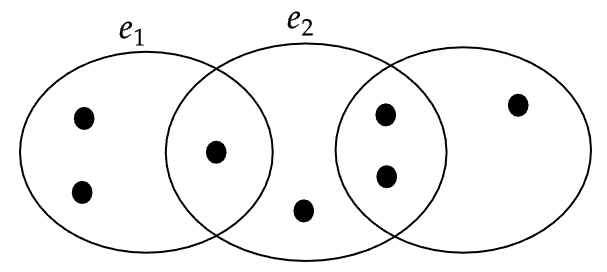}
        \caption{A hypergraph that is not wide-path connected.}
    \end{subfigure}
    \end{tabular}
    \caption{It is straightforward to show that subfigure (A) illustrates a wide-path connected hypergraph. Subfigure (B) illustrates a hypergraph that is not wide-path connected as $|e_1\cap e_2|=1$. This implies that there is not a wide path between distinct vertices in $e_1$ and $e_2$.}
    \label{fig:WidePath}
\end{figure}

\begin{lemma}\label{lem:WidePathConnected}
    Let $H= (V,E)$ be a wide-path connected hypergraph. Then $\Lean(H)= \chi([H]_2)$.
\end{lemma}
\begin{proof}
    Let $C$ be any lean coloring of $H$. Let $e_1,e_n \in E$ be edges of $H$, and let $P: e_1,e_2\ldots,e_{n-1},e_n$ be a wide path between them. 
    Note that if $e_1$ is monocolored (respectively, multicolored) under $C$, then $e_2$ must be monocolored (resp., multicolored) under $C$, since $e_1$ and $e_2$ share at least two vertices. 
    We use induction to show that it must be true that every edge in the wide path $P$ is monocolored (resp., multicolored). 
    Assume edges $e_1$, $e_2$, $\ldots$,  $e_k$ in $P$ are all monocolored (resp., multicolored). 
    Suppose $e_k$ and $e_{k+1}$ are successive edges in the wide path $P$.
    Then, since $e_k$ and $e_{k+1}$ share at least two vertices, it must be the case that if $e_k$ is monocolored (resp., multicolored) under $C$, then $e_{k+1}$ is monocolored (resp., multicolored) under $C$. 

    Hence, if any edge in the wide path $P$ is monocolored under $C$, then so is every edge in the wide path $P$. 
    But, by assumption, we may connect any two vertices in $H$ with a wide path, so monocoloring any edge under $C$ would imply every edge is monocolored, and so there is only one color present in the coloring $C$. 
    This violates the rule that any lean coloring must use at least two colors, and so at least one edge must be multicolored under $C$. 
    By our reasoning above, every edge in $H$ must be multicolored under $C$. 

    Equivalently, any lean coloring $C$ requires that if two vertices are adjacent, then they must be assigned distinct colors. 
    It follows that every lean coloring of $H$ is a strong coloring of $H$. It is also true that every strong coloring of $H$ is also a lean coloring of $H$.
    Hence, in this case, the set of lean colorings of $H$ coincides with the set of strong colorings of $H$.
    As proven in \cite[p.\ 116]{Hypergraphs}, the strong chromatic number of $H$ is equal to the chromatic number of $[H]_2$. 
    Thus, we conclude that the lean number of $H$ is also the chromatic number of the $2$-section of $H$.
\end{proof}

In light of \Cref{lem:WidePathConnected}, next we consider hypergraphs with a ``relatively large" number of edges, which have a higher lean number. 
Consider a hypergraph $H=(V,E)$ and let $3 \leq r \leq |V|$. Then the hypergraph $H$ is $r$-\textit{complete}, if for each set of $r$ of its vertices, $v_1,v_2,\ldots,v_r \in V$, there exists an edge $e \in E$ that contains $v_1,v_2,\ldots,v_r$ (and possibly more). 

\begin{corollary}
\label{cor:rcomp}
An $r$-complete hypergraph $H = (V,E)$, with $3\leq r\leq |V|$, satisfies $\Lean(H)=|V|$. 
\end{corollary}

 \begin{proof}
    We begin by establishing that an $r$-complete hypergraph is wide-path connected. To see this, let $e_1,e_2 \in V$. Let $v,v' \in e_2$. Pick any edge $e' \in E$ that contains $v$ and satisfies $|e' \cap e_1 | = r-1 \geq 2$. Then pick any edge $e'' \in E$ that contains both $v$ and $v'$ and that satisfies $|e' \cap e''| = r-1 \geq 2$. Then $\{v,v'\} \subset e'' \cap e_2$ and so $e_1,e',e'',e_2$ is a wide path from $e_1$ to $e_2$.
    \Cref{lem:WidePathConnected} then applies, and so $\Lean(H) = \chi([H]_2)$. 
    Since each pair of vertices in $H$ are incident to some edge, $[H]_2$ is simply the complete graph on $|V|$ vertices, which has chromatic number $\chi([H]_2) = |V|$. 
\end{proof}

Before we present more results, we recall that for a hypergraph $H = (V,E)$, a set $A \subseteq V$ is an \textit{articulation set of $H$} \cite[Page~71]{HypTheory} if:
\begin{enumerate}
    \item there are two edges $e_1,e_2 \in E$ such that $A = e_1 \cap e_2$, and 
    \item the induced subhypergraph $H(V\setminus A)$ is disconnected and has a \textit{strictly greater} count of components than $H$ or is the empty hypergraph. 
\end{enumerate}
Condition 2 states that the subhypergraph $H(V\setminus A)$ removes the vertices in $A$ from $H$, but does not remove the edges incident to the vertices that are removed. 

For our purposes, we focus on articulation sets of hypergraphs that are singletons. We call such sets \textit{articulation vertices}. 
We introduce another operation on a hypergraph and a specified articulation vertex $v$ of that hypergraph, which we define as a \textit{cleave of $H$ at $v$}. 

\begin{definition}\label{def:cleave}
    Let $H=(V,E)$ be a hypergraph with $|E| = n$ and articulation vertex $v \in V$. The \textit{cleave~of~$H$~at~$v$}, denoted $\cleave(H,v)$, is the subhypergraph $H(V \setminus\{v\})$, with the following additional rule: for each $1 \leq i \leq n$, if $v \in e_i$, then we add into $e_i$ the new vertex $v_i$. 
\end{definition}

\Cref{fig:cleave} displays an example of a cleave operation on a hypergraph.

\begin{figure}[h!]
    \centering
    \begin{tabular}{rr}
    \begin{subfigure}[b]{0.3\textwidth}
        \includegraphics[width=\textwidth]{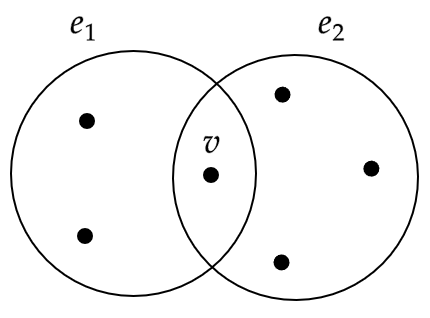}
        \caption{Hypergraph $H$}
    \end{subfigure}
    &\qquad\qquad
    \begin{subfigure}[b]{0.4\textwidth}
\includegraphics[width=\textwidth]{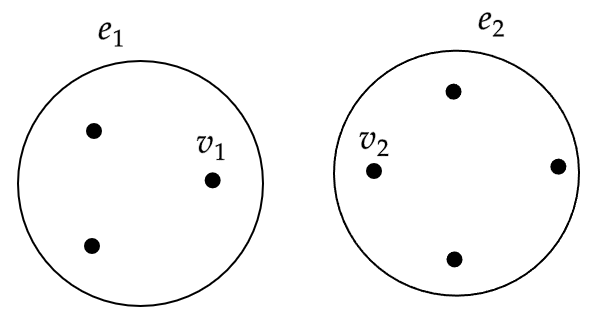}
        \caption{$\cleave(H,v)$}
    \end{subfigure}
    \end{tabular}
    \caption{Subfigure (A) is the initial hypergraph $H$ and we identify the articulation vertex $v$, which is incident to edges $e_1$ and $e_2$. 
    Subfigure (B) illustrates the cleave of $H$ at the articulation vertex $v$. Note that the vertex $v$ has been replaced with two vertices $v_1$ and $v_2$, with one appearing in each edge of $H$ where $v$ was.}
    \label{fig:cleave}
\end{figure}

As the next result shows, performing a cleave operation on a hypergraph, or \textit{cleaving} it, in this way establishes a relationship between the lean number of a hypergraph and the lean numbers of the \textit{components} of its cleave at an articulation vertex.

\begin{lemma}\label{lem:separation}
    Let $H = (V,E)$ be a CR hypergraph and let $v \in V$ be an articulation vertex of $H$. Also let $H_1,H_2,\ldots,H_n$ be the components of $\cleave(H,v)$. Then $\Lean(H) = \min_{1 \leq i \leq n}\Lean(H_i)$. 
\end{lemma}

\begin{proof}
    First, we prove that $\Lean(H) \leq \min_{1 \leq i \leq n}\Lean(H_i)$ by considering lean colorings on the components of $\cleave(H,v)$. 
    
    First we show that $H$, being connected and reduced, cannot have edges containing a single vertex, i.e., singleton edges.
    If $H$ had a singleton edge $e$ and $e$ were contained in another edge distinct from $e$, then $H$ would not be reduced. If $e$ were not contained in any edge distinct from $e$, then $H$ would not be connected. 
    Hence, each edge of a connected and reduced hypergraph contains at least two vertices. 
    It follows that the vertex set of $H_i$ for each $i$ contains at least two vertices.
    
    Without loss of generality, we assume $\Lean(H_1) \leq \Lean(H_i)$ for $2 \leq i \leq n$. Let $H_1$ have vertex set $V_1$. 
    Note $\Lean(H_1)$ is well-defined since by our previous paragraph we know $H_1$ has at least two vertices.
    Let $L' \in \mathcal{L}(H_1)$ be an optimal lean coloring of $H_1$, and let $L$ be a lean coloring of $\cleave(H,v)$ defined as follows for the vertices present in $\cleave(H,v)$:

    \begin{align*}
        L(w)&= \begin{cases}
            L'(w) &\text{if $w \in V_1$}\\
            L'(v_1) &\text{otherwise}
        \end{cases}
    \end{align*}
where we recall that $v_1$ is the copy of the vertex $v$ in $H_1$.
   In words, we are simply lean coloring the vertices of $H_1$ optimally (keeping the coloring used in $L'$ on this component), then coloring all vertices in the other components with the color of the vertex $v_1$. 
   Since the components $H_2,H_3,\ldots,H_n$ contain only monocolored edges and because $L'$, the lean coloring of $H_1$, uses at least two colors, we have that the coloring $L$ of $\cleave(H,v)$ is a lean coloring. \Cref{fig:HypCleavedColored} shows this coloring $L$ on the example in \Cref{fig:cleave}.
    
    \begin{figure}[h]
        \centering
        \includegraphics[width=0.5\linewidth]{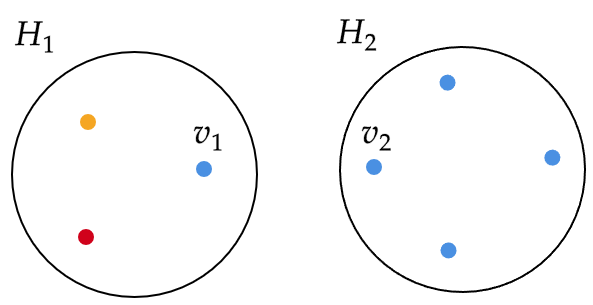}
        \caption{The hypergraph from subfigure (B) of \Cref{fig:cleave} has a coloring $L$ that displays an optimal lean coloring of the left component, denoted $H_1$. Since the cleaved vertex $v$ is colored blue (with its respective copy in each component labeled $v_1$ and $v_2$), each vertex contained in the component $H_2$ is colored blue.} 
        \label{fig:HypCleavedColored}
    \end{figure}
    
    We may use the lean coloring $L$ of $\cleave (H,v)$ to construct a lean coloring for $H$ by collapsing each $v_i$ onto $v$, and in doing so we recover the hypergraph $H$. 
    Hence, we construct $\tilde{L}$, a coloring of $H$, defined by:
    \begin{align*}
        \tilde{L}(w)&= \begin{cases}
            L'(w) &\text{if $w \in V_1\setminus\{v_1\}$}\\
            L'(v_1) &\text{if $w = v$ or $w \in V \setminus V_1$.}
        \end{cases}
    \end{align*}

    Specifically, $\tilde{L}$ colors each vertex that belonged to $H_1$ according to the coloring $L'$. We then color the cleaved vertex $v$ and each vertex not belonging to $H_1$ in the same color
    $L'(v_1)$. \Cref{fig:HypUncleavedColored} shows our example hypergraph when colored with $\tilde{L}$.
    
    \begin{figure}[h]
        \centering
        \includegraphics[width=0.4\linewidth]{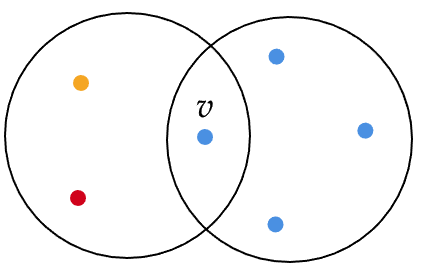}
        \caption{We use the colorings $L$ and $L'$ to lean color the vertices of $H$ via $\tilde{L}$.}
        \label{fig:HypUncleavedColored}
    \end{figure}
    
    We have now produced a lean coloring of $H$ using $\Lean(H_1)$ colors, and so $\Lean(H) \leq \Lean(H_1)$. \\
    
    We now prove that $\Lean(H) \geq \min_{1 \leq i \leq n}\Lean(H_i)$. Let $v$ be an articulation vertex of $H$ and let $C$ be an optimal lean coloring of $H$.
    Since $H$ is connected, there must be at least one edge $e \in E$ that is multicolored under $C$ (for if there were not, $H$ would be disconnected or each vertex would be colored with only one color). 
    Then we define the following lean coloring $C'$ of the vertices of $\cleave(H,v)$:

    \begin{align*}
        C'(w) &= \begin{cases}
            C(w) &\text{if $w \in V\setminus\{v\}$}\\
            C(v) &\text{if $w = v_i$ for some $1 \leq i \leq n$.}
        \end{cases}
    \end{align*}

    We wish to prove now that the edge $e$, multicolored under $C$, has a corresponding partner in $\cleave(H,v)$ that is multicolored under $C'$. If the edge $e$ did not contain the articulation vertex $v$, then $e$ is also multicolored under $C'$. 
    Since $H$ is connected and reduced, every edge contains at least 2 vertices.
    Hence,
    if the edge $e$ \textit{did} contain the cleaved vertex $v$, then its corresponding edge $e'=e\setminus\{v\} \cup \{v_1\}$ in $E'$ (the edge set of $\cleave(H,v)$)  contains at least two vertices and the edge $e'$ is also multicolored under~$C'$. Without loss of generality, suppose $e'$ belongs to the edge set of the component $H_1 = (V_1,E_1)$ in $\cleave(H,v)$. We have now created a lean coloring of $H_1$, i.e., $C'|_{V_1}$, that uses at most $\Lean(H)$ colors. Hence, $\Lean(H) \geq \Lean(H_1) \geq \min_{1 \leq i \leq n}\Lean(H_i).$
    \end{proof}

\begin{remark}
    \Cref{lem:WidePathConnected} states that a wide-path connected hypergraph $H$ has lean number $\Lean(H) =~\chi([H]_2)$. It may be of interest to the reader to know that the converse is false, see hypergraph $H$ in \Cref{fig:RemarkAboutWPCAandArtVerts}. Additionally, a wide-path connected hypergraph has no articulation vertices. The converse of this is also false, see hypergraph $P$ in \Cref{fig:RemarkAboutWPCAandArtVerts}. There is no relationship between a hypergraph $E$ having no articulation vertices and it satisfying $\Lean(H) = \chi([H]_2)$. For visual examples, see \Cref{fig:RemarkAboutWPCAandArtVerts}. 
\end{remark}

\begin{figure}[h!]
    \centering
    \begin{tabular}{rr}
    \begin{subfigure}[b]{0.3\textwidth}
        \includegraphics[width=.7\textwidth]{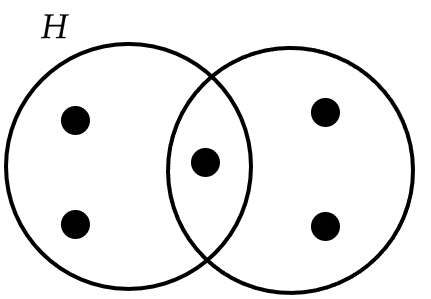}
    \end{subfigure}
    \begin{subfigure}[b]{0.25\textwidth}
\includegraphics[width=.7\textwidth]{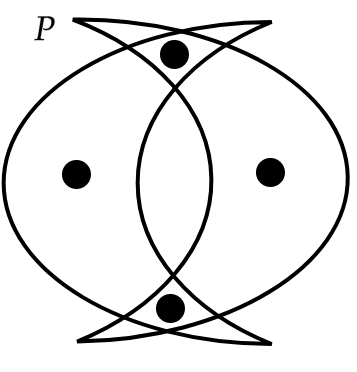}
    \end{subfigure}
    \begin{subfigure}[b]{0.25\textwidth}
        \includegraphics[width=.7\textwidth]{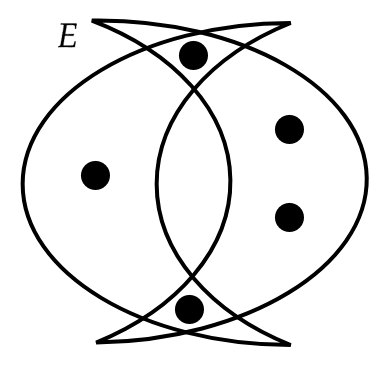}
    \end{subfigure}
    \end{tabular}
    \caption{Hypergraph $H$ satisfies $\Lean(H) = \chi([H]_2)$ = 3, is not wide-path connected, and has an articulation vertex. Hypergraph $P$ has no articulation vertices, is not wide-path connected, and does satisfy $\Lean(P) = \chi([P]_2)=3$. Hypergraph $E$ has no articulation vertices, it not wide-path connected, and it does not satisfy $\Lean(H) = \chi([H]_2)$ (since $\Lean(E) = 3$ and $\chi([E]_2) = 4$).}
    \label{fig:RemarkAboutWPCAandArtVerts}
\end{figure}

Instead of requiring larger combinatorial techniques, we can transform our problem into a series of much smaller lean coloring problems via \Cref{lem:separation}. It is rarely the case that an arbitrary hypergraph has exactly one articulation vertex. Iterated use of \Cref{lem:separation} proves valuable, though we may worry that choosing to cleave a hypergraph at its articulation vertices in different orders may introduce ambiguity. In fact, we need not worry about this, as \Cref{thm:GoldenGoose} will demonstrate. We require \Cref{lem:TheHammer} before we establish \Cref{thm:GoldenGoose}.

\begin{lemma}\label{lem:TheHammer}
    Let $H = (V,E)$ be a reduced hypergraph with articulation vertex $v \in V$. Then each component of $\cleave(H,v)$ is reduced.
\end{lemma}

\begin{proof}
    Suppose for the sake of contradiction that $K$ is a component of $\cleave(H,v) = ((V\setminus\{v\}) \cup \{v_1,\ldots, v_n\},F)$ that is not reduced, i.e., there exist distinct edges $f_1,f_2 \in F$ such that $f_1 \subset f_2$. 
    We have the following cases:
    \begin{enumerate}
        \item The copy of $v$ in $K$, call it $v_k$, is not in $f_2$, and hence not in $f_1$.
        \item The vertex $v_k$ is in $f_1$ and hence it is in $f_2$.
        \item The vertex $v_k$ is in $f_2$, but it is not in $f_1$.
    \end{enumerate}

    For Case 1:
    If $v_k \notin f_2$, then $f_1$ and $f_2$ are edges in $E$, and so $f_1$ cannot be a subset of $f_2$, as $H$ is reduced, reaching a contradiction. 

    For Case 2:
    This case is impossible, since by definition of $\cleave(H,v)$, the vertex $v_k$ is contained in exactly one edge, but we assumed $v_k \in f_1$ and $f_1\subset f_2$, hence $v_k \in f_2$, a contradiction. 

    For case 3:
    Finally, consider $v_k \notin f_1$ but $v_k \in f_2$. 
    Then $f_1$ is an edge in $E$, and in $H$ there exists an edge $e\coloneqq (f_2 \setminus \{v_k\}) \cup \{v\}\in E$, (i.e., the original version of the edge $f_2$ prior to the cleave) which contains $f_1 \in e$. This contradicts the assumption that $H$ is reduced.
\end{proof}

\Cref{lem:TheHammer} implies that each time we cleave a reduced hypergraph at an articulation vertex, each component is a CR hypergraph. In order to state \Cref{thm:GoldenGoose} more plainly, we introduce the following notation.

\begin{definition}
    Let $H = (V,E)$ be a reduced hypergraph with distinct articulation vertices $v_1,v_2,\ldots,v_n \in~V$. Let $H_1 \coloneqq \cleave(H,v_1)$. Then let $H_2$ be the hypergraph obtained by performing a cleave at $v_2$ on the (CR) component of $H_1$ containing $v_2$, and denote the resulting hypergraph $H_2 \coloneqq \cleave(H,(v_1,v_2))$. 
    Continue in this way until we have performed a cleave at $v_n$ on the (CR) component of $H_{n-1}$ containing~$v_n$. The resulting hypergraph is denoted $\cleave(H,(v_1,v_2,\ldots,v_n))$.
\end{definition}

\begin{theorem}\label{thm:GoldenGoose}
    Let $H = (V,E)$ be a CR hypergraph with distinct articulation vertices $v_1,v_2,\ldots,v_n \in V$. Let the components of $\cleave(H,(v_1,v_2,\ldots,v_n))$ be $K_1,K_2,\ldots,K_m$ for some $m \in \mathbb{N}$. Then $\Lean(H) = \min_{1 \leq i \leq m} \Lean(K_i)$. 
\end{theorem}

\begin{proof}
    We proceed by induction to show that the components of $H_j\coloneqq\cleave(H,(v_1,v_2,\ldots,v_j))$ have a similar property for each $1 \leq j \leq n$. When $j=1$, the previous statement is simply \Cref{lem:separation}. 
    Assume that the statement is true for each $1 \leq j \leq n-1$. In other words, if $H_j^1, H_j^2, \ldots, H_j^{l_j}$ are the components of $H_j$, then $\Lean(H) = \min_{1\leq x\leq l_j} \Lean(H_j^{x})$. 
    This assumption implies that $1.$ there exists a component of $H_{n-1}$, say $H_{n-1}'$, with $\Lean(H) = \Lean(H_{n-1}')$, and $2.$ each other component of $H_{n-1}$ has lean number at least $\Lean(H)$. 
    If vertex $v_n$ is a member of $H_{n-1}'$, then \Cref{lem:separation} implies that there is a component of $\cleave(H,(v_1,v_2,\ldots,v_n))$ with lean number equal to $\Lean(H)$, and all other components have lean number at least $\Lean(H)$. 
    If $v_n$ is not a member of $H_{n-1}'$, then performing a cleave at $v_n$ preserves $H_{n-1}'$, thereby exhibiting a component of $H_n \coloneqq \cleave(H,(v_1,v_2,\ldots,v_n))$ with lean number equal to $\Lean(H)$.
    Additionally, if $K$ is the component of $H_{n-1}$ containing $v_n$, then after cleaving at $v_n$, each component of $\cleave(K,v_n)$ has lean number at least $\Lean(K) \geq \Lean(H)$. 
\end{proof}

\Cref{thm:GoldenGoose} implies that the lean number of a hypergraph is preserved in at least one component, independent of the sequence in which we cleave the articulation vertices. This becomes very useful in Algorithm~\ref{alg:CleaveAlg} in \Cref{sec:Algo}. 

Next, we consider the effect of the  cleave operation on the hypergraph's incidence matrix.
Recall that for a given hypergraph $H = (V,E)$, its \textit{(vertex-edge) incidence matrix} is a binary matrix of dimension $|V| \times |E|$, where the entry in the $m$th row and $n$th column is a $1$ if vertex $m$ is incident to edge $n$, and $0$ otherwise \cite{Parui_2025}. Each unlabeled hypergraph corresponds to a set of incidence matrices, with each matrix being identical up to a permutation of the rows and a permutation of the columns. Likewise, each set of binary matrices that are identical up to a permutation of rows and a permutation of columns corresponds to a unique unlabeled hypergraph. Next, we provide an equivalent definition of a cleave using incidence matrices.

\begin{definition}
    Let $H=(V,E)$ be a hypergraph with vertex-edge incidence matrix $M$ and articulation vertex $v \in V$. Suppose $v$ is incident to the edges $e_1,e_2,\ldots,e_k$. The \textit{cleave of $H$ at $v$}, denoted $\cleave(H,v)$, is the hypergraph corresponding to the binary matrix obtained in the following way:
    \begin{enumerate}
        \item Delete the row of $M$ corresponding to the vertex $v$.
        \item For each $1 \leq i \leq k$, insert into the matrix a row that has a $1$ in the $i$th column and $0$s elsewhere.  
    \end{enumerate}
\end{definition}

See \Cref{fig:cleaveincidence} for examples of how cleaving a hypergraph can be viewed equivalently as an operation on its incidence matrix.

\begin{figure}
    \centering
    \begin{tabular}{rr}
    \begin{subfigure}[b]{0.8\textwidth}
        \includegraphics[width=.9\textwidth]{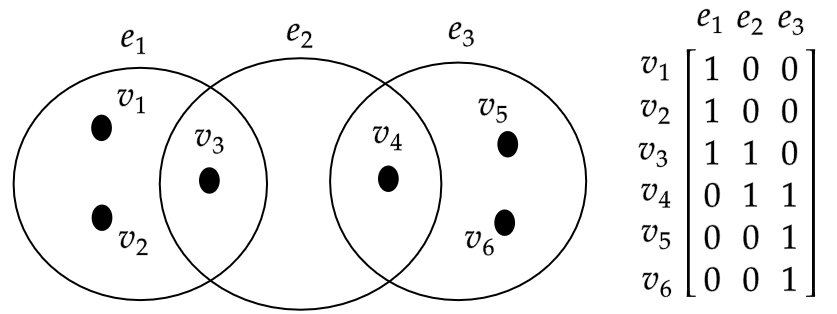}
        \caption{A hypergraph $H$ with articulation vertices $v_3$ and $v_4$. Its incidence matrix is pictured to the right.}
    \end{subfigure}
    \\
    \begin{subfigure}[b]{0.8\textwidth}
\includegraphics[width=\textwidth]{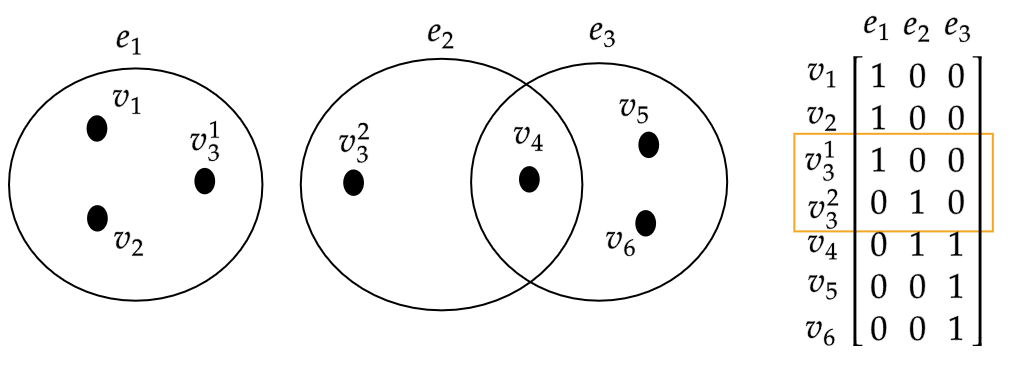}
        \caption{The hypergraph $\cleave(H,v_3)$ and its incidence matrix. In the orange box, we see that the row previously corresponding to vertex $v_3$ has been replaced by two rows, labeled $v_3^1$ and $v_3^2$. These two rows indicate which edges $v_3$ belongs to in $H$.}
    \end{subfigure}
    \\
    \begin{subfigure}[b]{0.8\textwidth}
\includegraphics[width=\textwidth]{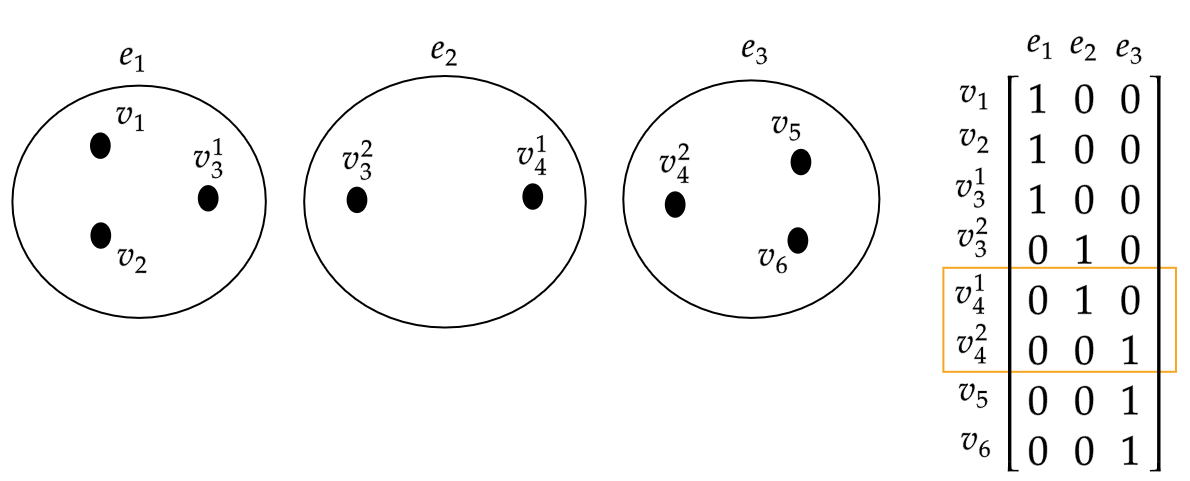}
        \caption{The hypergraph $\cleave(H,(v_3,v_4))$ and its incidence matrix. In the orange box, we see that the row in the matrix in (B) previously corresponding to vertex $v_4$ has been replaced by two rows, labeled $v_4^1$ and $v_4^2$. These two rows indicate which edges $v_4$ belongs to in $H$.}
    \end{subfigure}
    \end{tabular}
    \caption{}
    \label{fig:cleaveincidence}
\end{figure}

Two unlabeled hypergraphs $H_1$ and $H_2$ are isomorphic exactly when the incidence matrix $M_2$ of $H_2$ can be obtained by a sequence of row swaps and column swaps of the incidence matrix $M_1$ of $H_1$. Using this idea, we establish that the order in which we perform cleaves on a given hypergraph is, in a sense, inconsequential. 
As an example, \Cref{fig:cleaveincidence} illustrates that we arrive at the hypergraph illustrated in subfigure (C) regardless of whether we cleave at $v_3$, then $v_4$, or at $v_4$, then $v_3$. 

\begin{theorem}\label{thm:DiamondGoose}
    Let $H = (V,E)$ be a CR hypergraph and $v,w \in V$ be articulation vertices. Then $\cleave(H,(v,w))$ is isomorphic to $\cleave(H,(w,v))$. 
\end{theorem}

\begin{proof}
    Let $M_{v,w}$ and $M_{w,v}$ be the incidence matrices of $\cleave(H,(v,w))$ and $\cleave(H,(w,v))$, respectively. We proceed to show that $M_{v,w}$ and $M_{w,v}$ are identical after a permutation of the rows. Both matrices have $|E|$ columns and $|V|$ $+$ $\deg(v)$ $+$ $\deg(w)$ $ - $ $2$ rows. Both matrices have exactly the same rows:

    \begin{itemize}
        \item $|V| - 2$ rows describing the incidences of vertices in $V \setminus \{v,w\}$,
        \item $\deg(v)$ rows, each containing exactly one copy of $1$, and all other entries are $0$. Each edge $e\in E$ that contains $v$ in $H$ is represented by exactly one row containing a $1$ in the column representing $e$. 
        \item $\deg(w)$ rows, each containing exactly one copy of $1$, and all other entries are $0$. Each edge $e\in E$ that contains $w$ in $H$ is represented by exactly one row containing a $1$ in the column representing $e$. 
    \end{itemize}

    These sets of rows are disjoint and together comprise the rows of both $M_{v,w}$ and $M_{w,v}$. Since these two matrices contain the same rows, but perhaps in a different order, they describe isomorphic hypergraphs. 
\end{proof}

\Cref{thm:DiamondGoose} implies that regardless of the order in which we perform the cleaves of articulation vertices of a hypergraph, we always end with the same hypergraph, up to isomorphism. To finish this section, we relate graph coloring techniques to the lean number problem. 

\begin{lemma}\label{lem:2Section}
    For a hypergraph $H = (V,E)$, $\Lean(H) \leq \chi([H]_2)$. 
\end{lemma}
\begin{proof}
    Note that any strong coloring of $H$ is also a lean coloring of $H$, and so the lean number of $H$ is at most the strong chromatic number of $H$, which is equal to $\chi([H]_2)$.
\end{proof}

Since there are many simple bounds on chromatic numbers of graphs, \Cref{lem:2Section} is a powerful tool used to establish a bound for the lean number of an arbitrary hypergraph. For example, we have the following corollary.

\begin{corollary}
    Let $H= (V,E)$ be a hypergraph and let $\Delta$ be the greatest vertex degree of a vertex in~$V$. Then $\Lean(H) \leq \Delta + 1$. 
\end{corollary}

\begin{proof}
    Since $\Delta$ is the greatest degree of a vertex in $V$, it is also the greatest degree of a vertex in $[H]_2$. A greedy coloring of $[H]_2$ shows that $\chi([H]_2) \leq \Delta + 1$, and so $\Lean(H) \leq \chi([H]_2) \leq \Delta + 1$.
\end{proof}

In the next section, we use the results established here to detail an algorithm to give further bounds on the lean number of an arbitrary hypergraph.

\section{Algorithm}\label{sec:Algo}
This section includes a general algorithm to produce an upper bound of the lean number of an arbitrary hypergraph. 
Python-like code is provided alongside the description of each step in the algorithm. Let our hypergraph of interest be $H = (V,E)$. In order to speak of the complexity of this algorithm, let $|V| = n$ and $|E| = m$. For instructive purposes, we use the hypergraph in \Cref{fig:coloringexample} to illustrate the workings of each step in the algorithm.

\begin{figure}[h!]
\centering
\begin{minipage}{.5\textwidth}
  \centering
  \includegraphics[width=.45\linewidth]{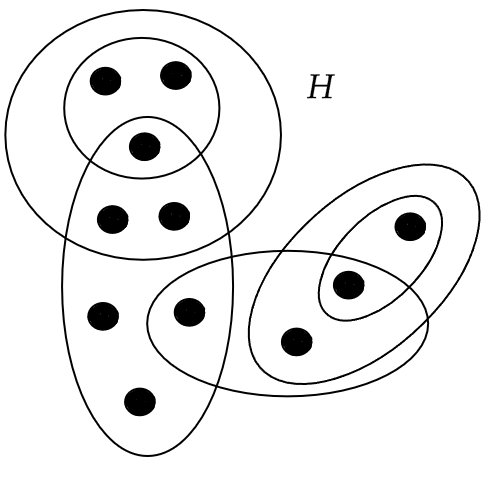}
  \captionof{figure}{A running example of a hypergraph $H$ that we lean color.}
    \label{fig:coloringexample}
\end{minipage}%
\begin{minipage}{.5\textwidth}
  \centering
  \includegraphics[width=0.45\linewidth]{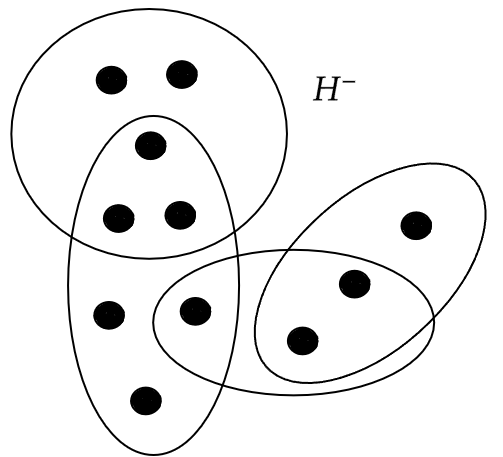}
    \captionof{figure}{Applying Algorithm~\ref{alg:RedAlg} to the hypergraph $H$ in \Cref{fig:coloringexample}, returns the illustrated reduced hypergraph $H^-$.
    }
    \label{fig:coloringexamplereduced}
\end{minipage}
\end{figure}

\subsection{Processing: Connectedness and Reduction}
\label{sec:proc}
As mentioned in \Cref{sec:special hypergraphs}, our problem is trivial when a graph is disconnected. 
If $H$ has two or more components, the lean number of $H$ is $2$ and the algorithm halts.
If not, we proceed to reduce $H$ in order to decrease the number of edges to consider. 
Recall, a hypergraph and its reduction have the same lean number, as proven in \Cref{lem:HypAndItsRed}. 
This step can be performed in $O(n+m)$ steps via a depth-first search. 

The reduction step can be performed in $O(m^2)$ steps at worst, as we must check every pair of edges. 
Below is our implementation of a reduction algorithm, inspired by the one present in \cite[Page~4]{HypTheory}.

\begin{algorithm}
    \SetKwInOut{Input}{Input}
    \SetKwInOut{Output}{Output}

    \underline{function ReductionAlgorithm}$(H)$\;
    \Input{Hypergraph $H=(V,E)$. In the Python package HypergraphX, $E$ is by default a set of tuples. For our purposes, we consider them as sets.}
    R = \{\}\;
    \For{$(i, s) \in$ \text{enumerate}($E$)}{ 
        \If{not \text{any}($s$ $\subset$ \text{other} for $(j,\text{other}) \in$ \text{enumerate}($E$) if $i$ $\neq$ $j$)}{Append $E[i]$ to $R$
        }
    }
    return $H^- = (V,R)$
    \caption{Reduction Algorithm}\label{alg:RedAlg}    
\end{algorithm}

Algorithm~\ref{alg:RedAlg} iterates over each edge and keeps it in the result set $R$ if it is not a subset of any other edge. What is returned is the hypergraph with only the kept edges, which is the reduction of $H$. The reduction of our example hypergraph is shown in \Cref{fig:coloringexamplereduced}. Now that we have performed our processing, we may proceed to simplify the problem further by cleaving the hypergraph along its articulation vertices, which we do next. 

\subsection{Locating and Cleaving Along Articulation Vertices}
In order to convert the problem of lean coloring the hypergraph $H$ into smaller lean coloring problems, we employ \Cref{lem:separation}. \Cref{fig:CleaveAlgo} shows the components we produce when we perform a cleave at the articulation vertex $v$.

\begin{figure}[h!]
    \centering
    \begin{tabular}{rr}
    \begin{subfigure}[b]{0.4\textwidth}
    \centering
        \includegraphics[width=.5\textwidth]{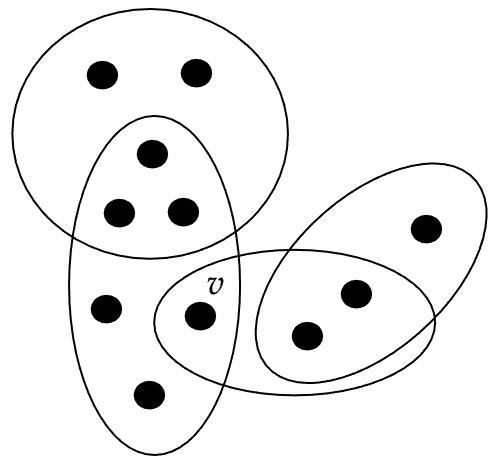}
        \caption{Hypergraph $H^-$ and vertex $v$.}\label{H and v}
    \end{subfigure}
    &\qquad\qquad
    \begin{subfigure}[b]{0.4\textwidth}
    \centering
\includegraphics[width=.65\textwidth]{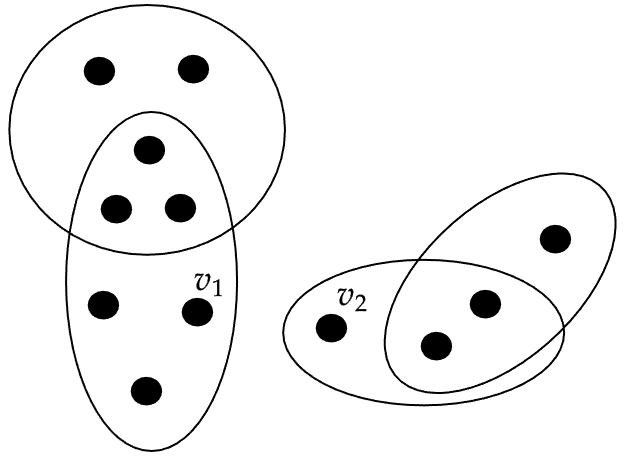}
        \caption{$\cleave(H^-,v)$}
    \end{subfigure}
    \end{tabular}
    \caption{Subfigure (A) is the initial hypergraph $H$ and we identify the articulation vertex $v$. 
    Subfigure (B) illustrates the cleave of $H$ at the articulation vertex $v$. Each edge that contains $v$ in $H$ now contains a copy of $v$, with labels $v_1$ and $v_2$.}
    \label{fig:CleaveAlgo}
\end{figure}

In order to locate the articulation vertices of a given hypergraph, we use Algorithm~\ref{alg:AVAlg}.

\begin{algorithm}[h!]
\SetKwInOut{Input}{Input}
\SetKwInOut{Output}{Output}

\underline{function FindArtVerts}$(H)$\;

\Input{A hypergraph $H=(V,E)$.}

ArtVerts = \{\}\;
\For{$v \in V$}{
    \If{\text{Hypergraph}$(V \setminus \{v\})$ is not connected}{
    Append $v$ to ArtVerts}    
    }
return ArtVerts

\caption{Articulation Vertex Location Algorithm}
\label{alg:AVAlg}

\end{algorithm}

Testing whether a hypergraph with a particular vertex removed is connected can be efficiently performed with the HypergraphX Python package. We use Algorithm~\ref{alg:AVAlg} in conjunction with Algorithm~\ref{alg:CleaveAlg}, which performs a cleave at each articulation vertex.

\begin{algorithm}
\SetKwInOut{Input}{Input}
\SetKwInOut{Output}{Output}

\underline{function CleavingAlgorithm}$(H,S)$\;

\Input{A hypergraph $H=(V,E)$. Set $S \coloneqq$ FindArtVerts($H$).}

$\text{LabelCounter} = \max(V) + 1$\;

F = E

\For{$v \in S$}{
    \For{$(i,\text{edge}) \in \text{enumerate}(F)$}{
        \If{$v \in \text{edge}$}{
            $edge[v] \leftarrow \text{LabelCounter}$\;
            $\text{LabelCounter} \leftarrow \text{LabelCounter} + 1$\;
        }
    }
}
Return $(\bigcup F,F)$

\caption{Cleaving Algorithm}
\label{alg:CleaveAlg}

\end{algorithm}

Algorithm~\ref{alg:CleaveAlg} first finds the smallest positive integer label not present on any vertex. We clone the edge set $E$ and call it $F$ so that we can produce the final edge set without changing $E$. 
It then enumerates each edge in $F$. If an articulation vertex $v$ is present in an edge, then it gets replaced with the next available label, i.e., the next smallest positive integer not present on any vertex. 
Hence, a cleave is performed at each occurrence of the articulation vertex $v$. 
This process is repeated for each articulation vertex, until a cleave is performed at each one. 
This produces an updated set of edges, which describes a hypergraph that is the output of Algorithm~\ref{alg:CleaveAlg}.

\subsection{Two-Sections}
Thanks to \Cref{lem:separation}, instead of lean coloring $H$ (or $H^-$), we may consider the lean colorings of the components of $\cleave(H^-,v)$, i.e., $H_1$ and $H_2$. 
Finding the component(s) with the smallest lean number or smallest bound on its lean number, we then color the vertices in the rest of the components a single color. Finding such a component can be done efficiently for this hypergraph, but not in general. Instead, an arbitrary hypergraph's lean number can be bounded by finding bounds on the lean numbers of components after performing any number of cleaves. 
A simple bound for a component's lean number is found with its $2$-section via \Cref{lem:2Section}. 
Locating a hypergraph's $2$-section can be performed in $O(nm)$ steps at worst. 
\Cref{fig:two section} displays the $2$-section of $\cleave(H^-,v)$. 

\begin{figure}[h!]
    \centering
    \begin{tabular}{rr}
    \begin{subfigure}[b]{0.3\textwidth}
        \includegraphics[width=\textwidth]{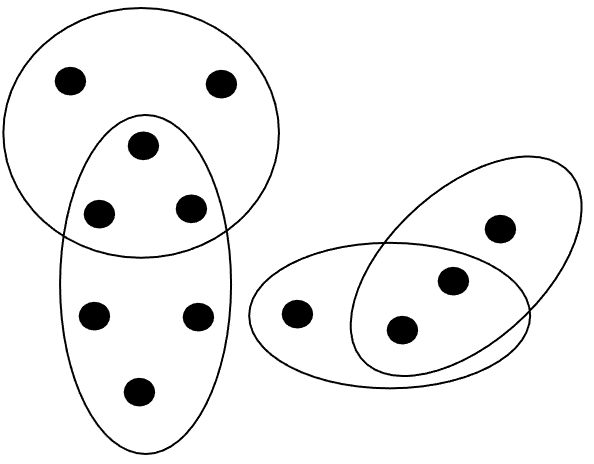}
        \caption{}
    \end{subfigure}
    &\qquad\qquad
    \begin{subfigure}[b]{0.3\textwidth}
\includegraphics[width=\textwidth]{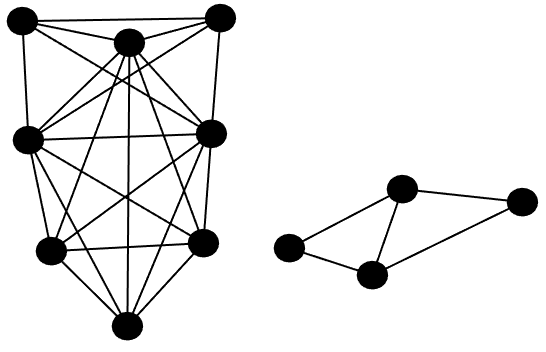}
        \caption{}
    \end{subfigure}
    \end{tabular}
    \caption{Subfigure (A) is $\cleave(H^-,v)$.
    Subfigure (B) is the $2$-section of $\cleave(H^-,v)$.}
    \label{fig:two section}
\end{figure}

Below is our implementation of a command that produces the two-section of a hypergraph by utilizing its adjacency matrix. Recall that the adjacency matrix $A$ of a hypergraph $H=(V,E)$ has dimension $|V| \times |V|$ and has entries $A_{ij}$, where $A_{ij}$ is the number of edges that contain both vertices $i$ and $j$. 
The matrix $A$ can be efficiently produced with the HypergraphX Python package. 
\begin{algorithm}
\SetKwInOut{Input}{Input}
\SetKwInOut{Output}{Output}

\underline{function TwoSectionAlgorithm}$(A)$\;

\Input{A hypergraph's adjacency matrix $A$ that has $n$ rows and $n$ columns.}
\For{$i \in \{1, 2, \ldots,n\}$}{
    \For{$j \in \{i,i+1,i+2,\ldots,n\}$}{
        \If{$A_{ij} >0$}{
        Append $(i,j)$ to TwoSection
        }
    }
}
return \text{Graph(TwoSection)}

\caption{$2$-Section Algorithm}
\label{alg:TSAlg}

\end{algorithm}

Algorithm~\ref{alg:TSAlg} takes in a hypergraph's adjacency matrix. 
Since the adjacency matrix of any hypergraph is symmetric, we only need to parse the upper half of the adjacency matrix, i.e., the entries $(i,j)$ with $i\leq j$.
If there is a positive entry in position $(i,j)$ of the adjacency matrix (i.e., vertex $i$ is adjacent to vertex $j$), then the output graph, \textit{TwoSection}, receives an edge between vertices $i$ and $j$. With these algorithms implemented, we are ready to color the hypergraph. 

\subsection{It is Always Graph Coloring Deep Down}
In light of \Cref{thm:GoldenGoose}, once we have performed a cleave at each articulation vertex of a hypergraph (in any order), we are left with a hypergraph that is expressible as a union of components. Finding a bound for the lean number of each component via \Cref{lem:2Section} yields a bound for the lean number for the original hypergraph. \Cref{fig:colored example final} displays a chromatic coloring of the two components in our example from \Cref{fig:two section}. Since the fewest number of colors used in either component is $3$, we conclude that an upper bound for the lean number of our original hypergraph $H$ is $3$. In fact, by inspection, this is optimal.

\begin{figure}[h!]
    \centering
    \begin{tabular}{rr}
    \begin{subfigure}[b]{0.4\textwidth}
        \includegraphics[width=\textwidth]{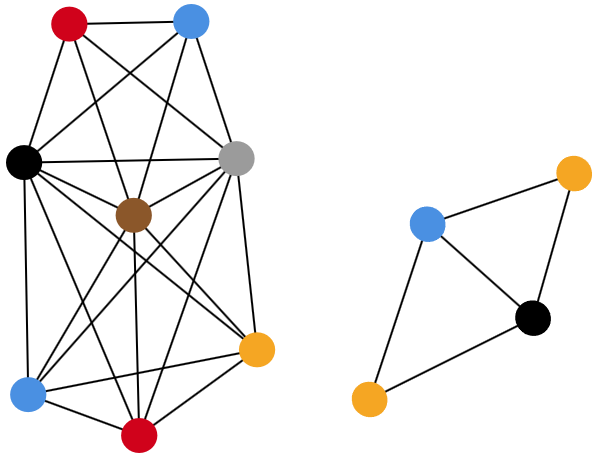}
        \caption{}
    \end{subfigure}
    &\qquad\qquad
    \begin{subfigure}[b]{0.3\textwidth}
\includegraphics[width=\textwidth]{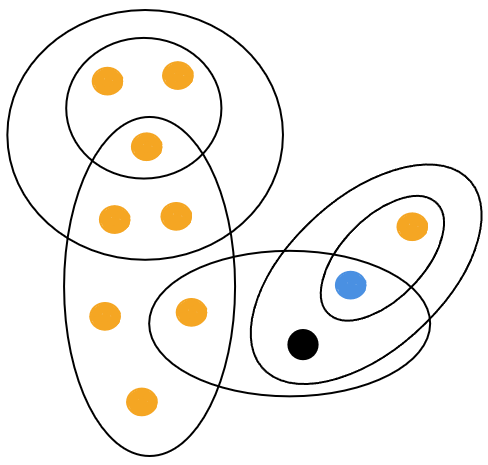}
        \caption{}
    \end{subfigure}
    \end{tabular}
    \caption{Subfigure (A) features a chromatic coloring of each component. Since the component with fewer vertices uses fewer colors (i.e., $3$), that is our upper bound for the lean number of our original hypergraph.
    Subfigure (B) is a lean coloring of our hypergraph $H$ using only three colors, which is optimal by inspection.}
    \label{fig:colored example final}
\end{figure}

Calculating the chromatic number for all but the simplest graphs is not possible in polynomial time. Certain techniques like dynamic programming can decide $k$-colorability in $O( (1+ \sqrt[3]{3})^n)$ time \cite{chromaticcomplexity}. Indeed, the problem in general is NP-Hard, making the overall problem of lean coloring at least NP-Hard in general. 
Consequently, the full algorithm proposed here has a complexity of $O(n + m + m^2 + 2nm + n^n) = O(n^n)$. 
For our purposes, we utilized the GCol library \cite{GCol}, which implements coloring algorithms present in 
\cite{GraphCol}. 

\section{Future Work}

\subsection{More Special Classes of Hypergraphs}
We conclude with some direction for future work. 
\begin{question}\Cref{cor:kpar} gives a family of hypergraphs for which the bound in \Cref{lem:EdgesAtLeastK} is tight. Hence we ask, can we characterize the collection of hypergraphs for which $\Lean(H)=k$? 
\end{question}

\begin{question}
    In \Cref{sec:LNOSKAL}, we find the lean numbers of several knots and links. We notice that among knots, a lean number of $5$ is rare, but among links, a lean number of $5$ appears frequently. Are there infinitely many knots or links with a lean number of $5$? Are there knots or links with a lean number of $6$, or knots or links with an arbitrarily large lean number?
\end{question}

\subsection{Graph Operations}
Given two graphs $G$ and $H$, we recall that Vizing's conjecture \cite{Vizing} states that the domination number of the Cartesian product of 
$G\square H$ is greater than or equal to the product of the domination number of the two graphs. Motivated by this famous conjecture, we ask:

\begin{question}
    Can we define an operation $\times$ that can produce a hypergraph from two given hypergraphs $H_1$ and $H_2$, and simply relates the lean number of $H_1 \times H_2$ to the lean number of the constituent hypergraphs?
\end{question}

\subsection{A Poset Problem}
One can define a partial order on hypergraphs with the same vertex set, where $H=(V,E)\prec H'=(V,E')$ whenever 
\begin{enumerate}
    \item $E \subseteq E'$, and
    \item for each $e \in E$, there exists an $f \in E'$ such that $e \subseteq f$. 
\end{enumerate}
In such a poset, any hypergraph $H=(V,E)$ 
lies on a chain which also contains its reduction
$H^-$ and its accession 
$H^+$. It would be of interest to study this poset and its properties.

\section{Acknowledgments}
I wish to eagerly thank Dr.\ Peh Ng, who generously granted me a position as a research assistant in order to study this problem, which helped me realize my ambition to become a mathematician. Her encouragement, advising, and direction helped me to develop the pseudocode for the algorithm present in this article, as well as many of our results. I also thank Dr.\ Pamela E. Harris, who spent countless hours refining my writing, encouraging my pursuits, inspiring new results, and generally showing me how to be a model intellectual. Finally, I thank Dr.\ Phil Chang for organizing the Comprehensive, Applied, and Tangible Summer School at the University of Wisconsin Milwaukee in the summer of 2025, allowing Dr.\ Harris and I to develop the code into what it is. The code was also developed with much help from the HypergraphX Python library \cite{hgx} and NetworkX.

\bibliographystyle{unsrt}
\bibliography{Bib.bib}

\section{Appendix}

\subsection{Lean Numbers of Some Knots and Links}\label{sec:LNOSKAL}

Using the tools presented in this paper and some inspection, the lean numbers of several knots and links were deduced. They are presented below with Rolfsen notation. The unknot cannot be lean colored, and so it is omitted from the list. 
The only link we have found so far that cannot be lean colored is $2_1^2$, i.e., the Hopf link, and so that is omitted from \Cref{tab:LinkColors}. \Cref{fig:AppTable} presents pictures of some knots in \Cref{tab:KnotColors}, which gives their lean colors.

\begin{table}[h]
    \begin{tabular}{cc | cc | cc | cc | cc | cc}
        Knot & Lean &  Knot & Lean  &  Knot & Lean  &  Knot & Lean &  Knot & Lean &  Knot & Lean \\
        $3_1$ & 3 & $8_1$ & 4 & $8_{15}$ & 3 & $9_8$ & 4 & $9_{22}$ & 4 & $9_{36}$ & 4 \\
        $4_1$ & 4 & $8_2$ & 4 & $8_{16}$ & 4 & $9_9$ & 4 & $9_{23}$ & 3 & $9_{37}$ & 3\\
        $5_1$ & 5 & $8_3$ & 4 & $8_{17}$ & 4 & $9_{10}$ & 3 & $9_{24}$ & 3 & $9_{38}$ & 3\\
        $5_2$ & 5 & $8_4$ & 4 & $8_{18}$ & 4 & $9_{11}$ & 3 & $9_{25}$ & 4 & $9_{39}$ & 4\\
        $6_1$ & 3 & $8_5$ & 3 & $8_{19}$ & 3 & $9_{12}$ & 4 & $9_{26}$ & 4 & $9_{40}$ & 3\\
        $6_2$ & 5 & $8_6$ & 4 & $8_{20}$ & 3 & $9_{13}$ & 4 & $9_{27}$ & 4 & $9_{41}$ & 4\\
        $6_3$ & 5 & $8_7$ & 4 & $8_{21}$ & 3 & $9_{14}$ & 4 & $9_{28}$ & 3 & $9_{42}$ & 4\\
        $7_1$ & 4 & $8_8$ & 4 & $9_1$ & 3 & $9_{15}$ & 3 & $9_{29}$ & 4 & $9_{43}$ & 4\\
        $7_2$ & 4 & $8_9$ & 4 & $9_2$ & 3 & $9_{16}$ & 3 & $9_{30}$ & 4 & $9_{44}$ & 4\\
        $7_3$ & 4 & $8_{10}$ & 3 & $9_3$ & 4 & $9_{17}$ & 3 & $9_{31}$ & 4 & $9_{45}$ & 4\\
        $7_4$ & 3 & $8_{11}$ & 3 & $9_4$ & 3 & $9_{18}$ & 4 & $9_{32}$ & 4 & $9_{46}$ & 3\\
        $7_5$ & 4 & $8_{12}$ & 4 & $9_5$ & 4 & $9_{19}$ & 4 & $9_{33}$ & 4 & $9_{47}$ & 3\\
        $7_6$ & 4 & $8_{13}$ & 4 & $9_6$ & 3 & $9_{20}$ & 4 & $9_{34}$ & 3 & $9_{48}$ & 3\\
        $7_7$ & 3 & $8_{14}$ & 4 & $9_7$ & 4 & $9_{21}$ & 4 & $9_{35}$ & 3 & $9_{49}$ & 4\\
    \end{tabular}
    \caption{The lean numbers of eighty-four knots.}\label{tab:KnotColors}
\end{table}

\begin{table}[h]
    \begin{tabular}{cc | cc | cc | cc | cc }
        Link & Lean  &  Link & Lean  &  Link & Lean &  Link & Lean &  Link & Lean \\
        $0_1^2$ & 3 & $7_1^2$ & 4 & $8_1^2$ & 5 & $8_4^3$ & 4 & $8_9^2$ & 4\\
        $4_1^2$ & 4 & $7_1^3$ & 4 & $8_1^3$ & 4 & $8_5^2$ & 3 & $8_9^3$ & 5 \\
        $5_1^2$ & 5 & $7_2^2$ & 4 & $8_1^4$ & 5 & $8_5^3$ & 5 & $8_{10}^2$ & 4 \\
        $6_1^2$ & 3 & $7_3^2$ & 4 & $8_2^2$ & 3 & $8_6^2$ & 4 & $8_{10}^3$ & 5 \\
        $6_1^3$ & 3 & $7_4^2$ & 4 & $8_2^3$ & 4 & $8_6^3$ & 4 & $8_{11}^2$ & 4 \\
        $6_2^3$ & 5 & $7_5^2$ & 4 & $8_2^4$ & 4 & $8_7^2$ & 4 & $8_{12}^2$ & 4 \\
        $6_3^3$ & 4 & $7_6^2$ & 4 & $8_3^2$ & 4 & $8_7^3$ & 4 & $8_{13}^2$ & 4 \\
        $6_2^2$ & 5 & $7_7^2$ & 4 & $8_3^3$ & 5 & $8_8^2$ & 4 & $8_{14}^2$ & 4 \\
        $6_3^2$ & 3 & $7_8^2$ & 4 & $8_4^2$ & 4 & $8_8^3$ & 4 & $8_{15}^2$ & 4 \\
        &  &  &  &  &  &  & & $8_{16}^2$ & 4 \\
    \end{tabular}
    \caption{The lean numbers of forty-six links.}\label{tab:LinkColors}
\end{table}

\begin{figure}[h!]
    \centering
    \begin{tabular}{r}
    \begin{subfigure}[b]{0.1\textwidth}
        \includegraphics[width=\textwidth]{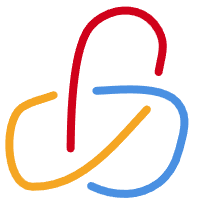}
        \caption*{$3_1$}
    \end{subfigure}
    \begin{subfigure}[b]{0.1\textwidth}
        \includegraphics[width=\textwidth]{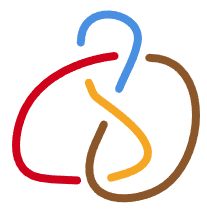}
        \caption*{$4_1$}
    \end{subfigure}
    \begin{subfigure}[b]{0.1\textwidth}
        \includegraphics[width=\textwidth]{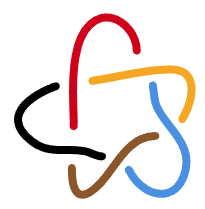}
        \caption*{$5_1$}
    \end{subfigure}
    \begin{subfigure}[b]{0.1\textwidth}
        \includegraphics[width=\textwidth]{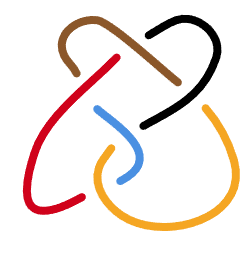}
        \caption*{$5_2$}
    \end{subfigure}
    \begin{subfigure}[b]{0.1\textwidth}
        \includegraphics[width=\textwidth]{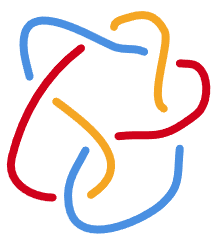}
        \caption*{$6_1$}
    \end{subfigure}
    \begin{subfigure}[b]{0.1\textwidth}
        \includegraphics[width=\textwidth]{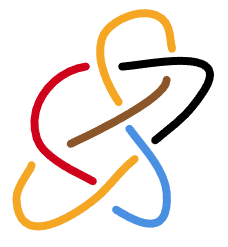}
        \caption*{$6_2$}
    \end{subfigure}
    \begin{subfigure}[b]{0.1\textwidth}
        \includegraphics[width=\textwidth]{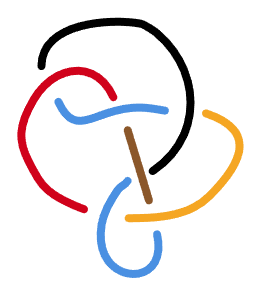}
        \caption*{$6_3$}
    \end{subfigure}
    \begin{subfigure}[b]{0.1\textwidth}
        \includegraphics[width=\textwidth]{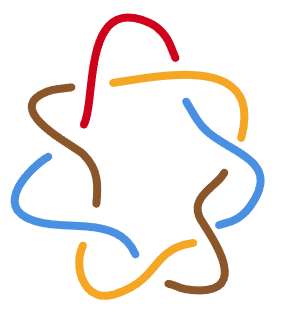}
        \caption*{$7_1$}
    \end{subfigure}
    \end{tabular}
    \begin{tabular}{r}
    \begin{subfigure}[b]{0.1\textwidth}
        \includegraphics[width=\textwidth]{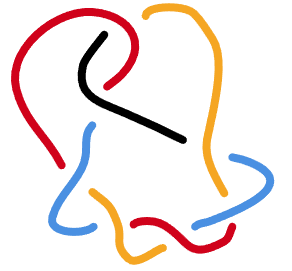}
        \caption*{$7_2$}
    \end{subfigure}
    \begin{subfigure}[b]{0.1\textwidth}
        \includegraphics[width=\textwidth]{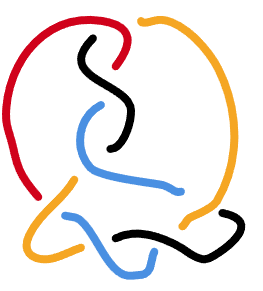}
        \caption*{$7_3$}
    \end{subfigure}
    \begin{subfigure}[b]{0.1\textwidth}
        \includegraphics[width=\textwidth]{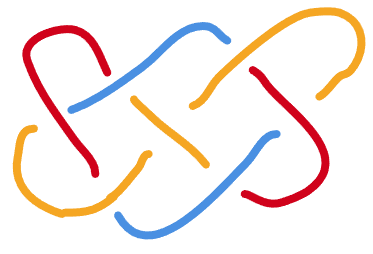}
        \caption*{$7_4$}
    \end{subfigure}
    \begin{subfigure}[b]{0.1\textwidth}
        \includegraphics[width=\textwidth]{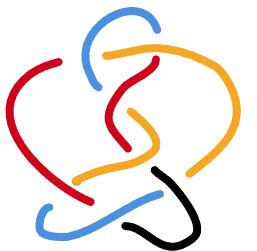}
        \caption*{$7_5$}
    \end{subfigure}
    \begin{subfigure}[b]{0.1\textwidth}
        \includegraphics[width=\textwidth]{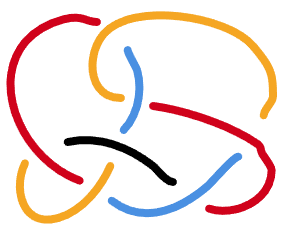}
        \caption*{$7_6$}
    \end{subfigure}
    \begin{subfigure}[b]{0.1\textwidth}
        \includegraphics[width=\textwidth]{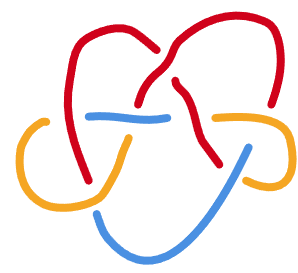}
        \caption*{$7_7$}
    \end{subfigure}
    \begin{subfigure}[b]{0.1\textwidth}
        \includegraphics[width=\textwidth]{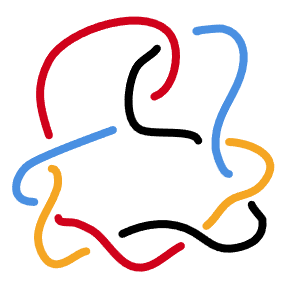}
        \caption*{$8_1$}
    \end{subfigure}
    \begin{subfigure}[b]{0.1\textwidth}
        \includegraphics[width=\textwidth]{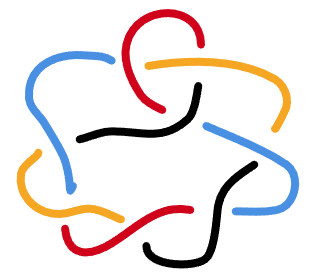}
        \caption*{$8_2$}
    \end{subfigure}
    \end{tabular}
    \begin{tabular}{r}
    \begin{subfigure}[b]{0.1\textwidth}
        \includegraphics[width=\textwidth]{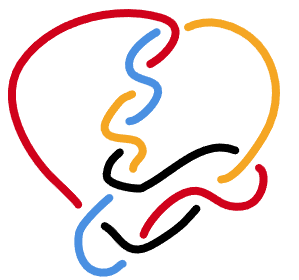}
        \caption*{$8_3$}
    \end{subfigure}
    \begin{subfigure}[b]{0.1\textwidth}
        \includegraphics[width=\textwidth]{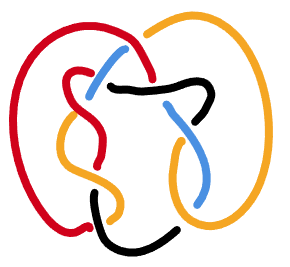}
        \caption*{$8_4$}
    \end{subfigure}
    \begin{subfigure}[b]{0.1\textwidth}
        \includegraphics[width=\textwidth]{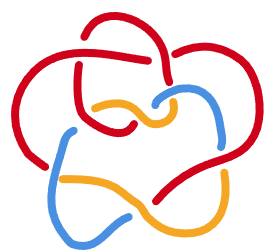}
        \caption*{$8_5$}
    \end{subfigure}
    \begin{subfigure}[b]{0.1\textwidth}
        \includegraphics[width=\textwidth]{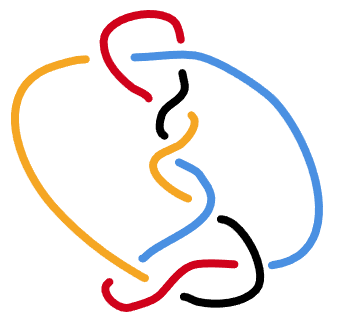}
        \caption*{$8_6$}
    \end{subfigure}
    \begin{subfigure}[b]{0.1\textwidth}
        \includegraphics[width=\textwidth]{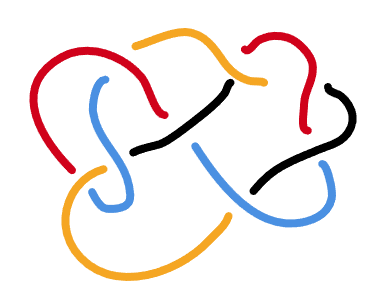}
        \caption*{$8_7$}
    \end{subfigure}
    \begin{subfigure}[b]{0.1\textwidth}
        \includegraphics[width=\textwidth]{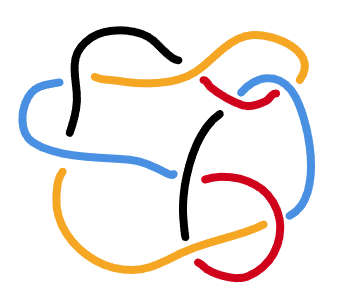}
        \caption*{$8_8$}
    \end{subfigure}
    \begin{subfigure}[b]{0.1\textwidth}
        \includegraphics[width=\textwidth]{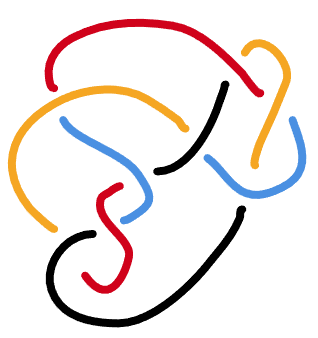}
        \caption*{$8_9$}
    \end{subfigure}
    \begin{subfigure}[b]{0.1\textwidth}
        \includegraphics[width=\textwidth]{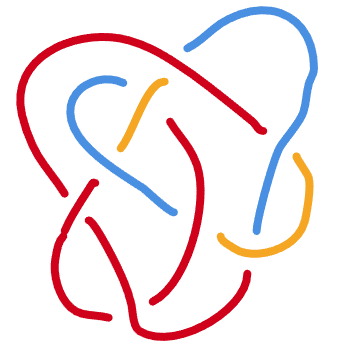}
        \caption*{$8_{10}$}
    \end{subfigure}
    \end{tabular}
    \caption{Twenty-four knots with optimal lean colorings.}\label{fig:AppTable}
\end{figure}

\end{document}